\documentclass[11pt,oneside,reqno]{article}

\usepackage{setspace}

\usepackage[a4paper,width=16cm,height=24cm]{geometry}

\usepackage{amsmath,amsthm,amssymb}
\usepackage{bbm}
\usepackage{mathrsfs}
\usepackage{enumitem}

\usepackage{booktabs, array}
\usepackage{makecell}
\usepackage{diagbox}
\newcolumntype{L}[1]{>{\raggedright\let\newline\\\arraybackslash\hspace{0pt}}m{#1}}
\newcolumntype{C}[1]{m{#1}}
\newcolumntype{R}[1]{>{\raggedleft\let\newline\\\arraybackslash\hspace{0pt}}m{#1}}

\usepackage{graphicx}
\usepackage[labelfont=bf, font=small]{caption}
\usepackage[format=hang]{subcaption}

\usepackage[authoryear,longnamesfirst]{natbib}
\usepackage{array}

\usepackage{filemod}
\usepackage[breaklinks=true,hidelinks]{hyperref}

\usepackage{etoolbox} 

\makeatletter

\def\@noindentfalse{\global\let\if@noindent\iffalse}
\def\@noindenttrue {\global\let\if@noindent\iftrue}
\def\@aftertheorem{%
  \@noindenttrue
  \everypar{%
    \if@noindent%
      \@noindentfalse\clubpenalty\@M\setbox\z@\lastbox%
    \else%
      \clubpenalty \@clubpenalty\everypar{}%
    \fi}}

\theoremstyle{plain}
\newtheorem{theorem}{Theorem}[section]
\AfterEndEnvironment{theorem}{\@aftertheorem}
\newtheorem{proposition}[theorem]{Proposition}
\AfterEndEnvironment{proposition}{\@aftertheorem}

\AfterEndEnvironment{lemma}{\@aftertheorem}

\AfterEndEnvironment{corollary}{\@aftertheorem}

\theoremstyle{definition}
\newtheorem{remark}[theorem]{Remark}
\newtheorem{definition}[theorem]{Definition}

\def\be#1{\begin{equation*}#1\end{equation*}}
\def\ben#1{\begin{equation}#1\end{equation}}

\def\besn#1{\begin{equation}\begin{split}#1\end{split}\end{equation}}

\def\ba#1{\begin{align*}#1\end{align*}}

\def\given{\typeout{Command 'given' should only be used within bracket command}}
\newcounter{@bracketlevel}
\def\@bracketfactory#1#2#3#4#5#6{
\expandafter\def\csname#1\endcsname##1{%
\addtocounter{@bracketlevel}{1}%
\global\expandafter\let\csname @middummy\alph{@bracketlevel}\endcsname\given%
\global\def\given{\mskip#5\csname#4\endcsname\vert\mskip#6}\csname#4l\endcsname#2##1\csname#4r\endcsname#3%
\global\expandafter\let\expandafter\given\csname @middummy\alph{@bracketlevel}\endcsname
\addtocounter{@bracketlevel}{-1}}%
}
\def\bracketfactory#1#2#3{%
\@bracketfactory{#1}{#2}{#3}{relax}{1mu plus 0.25mu minus 0.25mu}{0.6mu plus 0.15mu minus 0.15mu}
\@bracketfactory{b#1}{#2}{#3}{big}{1mu plus 0.25mu minus 0.25mu}{0.6mu plus 0.15mu minus 0.15mu}
\@bracketfactory{bb#1}{#2}{#3}{Big}{2.4mu plus 0.8mu minus 0.8mu}{1.8mu plus 0.6mu minus 0.6mu}
\@bracketfactory{bbb#1}{#2}{#3}{bigg}{3.2mu plus 1mu minus 1mu}{2.4mu plus 0.75mu minus 0.75mu}
\@bracketfactory{bbbb#1}{#2}{#3}{Bigg}{4mu plus 1mu minus 1mu}{3mu plus 0.75mu minus 0.75mu}
}
\bracketfactory{clc}{\lbrace}{\rbrace}
\bracketfactory{clr}{(}{)}
\bracketfactory{cls}{[}{]}
\bracketfactory{abs}{\lvert}{\rvert}
\bracketfactory{norm}{\Vert}{\Vert}
\bracketfactory{floor}{\lfloor}{\rfloor}
\bracketfactory{ceil}{\lceil}{\rceil}
\bracketfactory{angle}{\langle}{\rangle}

\newcounter{ctr}\loop\stepcounter{ctr}\edef\X{\@Alph\c@ctr}%
	\expandafter\edef\csname s\X\endcsname{\noexpand\mathscr{\X}}
	\expandafter\edef\csname c\X\endcsname{\noexpand\mathcal{\X}}
	\expandafter\edef\csname b\X\endcsname{\noexpand\boldsymbol{\X}}
	\expandafter\edef\csname I\X\endcsname{\noexpand\mathbbm{\X}}
	\expandafter\edef\csname r\X\endcsname{\noexpand\mathrm{\X}}
\ifnum\thectr<26\repeat

\let\@IE\IE\let\IE\undefined
\newcommand{\IE}{\mathop{{}\@IE}\mathopen{}}
\let\@IP\IP\let\IP\undefined
\newcommand{\IP}{\mathop{{}\@IP}}

\newcount\minute
\newcount\hour
\newcount\hourMins
\def\now{%
\minute=\time%
\hour=\time \divide \hour by 60%
\hourMins=\hour \multiply\hourMins by 60%
\advance\minute by -\hourMins%
\zeroPadTwo{\the\hour}:\zeroPadTwo{\the\minute}%
}
\def\zeroPadTwo#1{\ifnum #1<10 0\fi#1}

\numberwithin{equation}{section}
\allowdisplaybreaks[4]

\renewcommand\section{\@startsection {section}{1}{\z@}%
{-3.5ex \@plus -1ex \@minus -.2ex}%
{1.3ex \@plus.2ex}%
{\center\small\sc\mathversion{bold}\MakeUppercase}}

\def\subsection#1{\@startsection {subsection}{2}{0pt}%
{-3.5ex \@plus -1ex \@minus -.2ex}%
{1ex \@plus.2ex}%
{\bf\mathversion{bold}}{#1}}

\def\subsubsection#1{\@startsection{subsubsection}{3}{0pt}%
{\medskipamount}%
{-10pt}%
{\normalsize\itshape}{\kern-2.2ex. #1.}}

\def\blfootnote{\xdef\@thefnmark{}\@footnotetext}

\makeatother

\renewcommand{\cite}{\citet}

\def\^#1{\ifmmode {\mathaccent"705E #1} \else {\accent94 #1} \fi}
\def\~#1{\ifmmode {\mathaccent"707E #1} \else {\accent"7E #1} \fi}

\edef\-#1{\noexpand\ifmmode {\noexpand\bar{#1}} \noexpand\else \-#1\noexpand\fi}
\def\>#1{\vec{#1}}
\def\.#1{\dot{#1}}

\def\wt#1{\widetilde{#1}}
\def\atop{\@@atop}

\renewcommand{\leq}{\leqslant}
\renewcommand{\geq}{\geqslant}
\renewcommand{\phi}{\varphi}

\newcommand{\eq}{\eqref}

\newcommand{\law}{\mathscr{L}}

\newcommand{\Po}{\mathrm{Po}}
\def\sp#1{^{(#1)}}

\def\eqd{\stackrel{d}{=}}
\newcommand{\CRP}{\mathrm{CRP}}
\newcommand{\PD}{\mathrm{PD}}

\newcommand{\CF}{\mathrm{CF}}
\newcommand{\HM}{\mathrm{HM}}

\newcommand{\ICRP}{\mathrm{ICRP}}
\newcommand{\SICRP}{\mathrm{SICRP}}
\newcommand{\HICRP}{\mathrm{HICRP}}
\newcommand{\Dir}{\mathrm{Dir}}

\newcommand{\Gam}{\mathrm{Gamma}}

\newcommand{\samp}{\normalfont\textrm{\small SAMP}}

\newcommand{\frag}{\normalfont\textrm{\small FRAG}}
\def\fragd#1{\normalfont\textrm{\small FRAG}^{(#1)}}
\def\ifragd#1{\normalfont\textrm{\small I}\fragd{#1}}
\newcommand{\partit}{\normalfont\textrm{\small PART}}
\newcommand{\compos}{\normalfont\textrm{\small COMP}}

\def\coagd#1{\normalfont\textrm{\small COAG}^{(#1)}}
\def\icoagd#1{\normalfont\textrm{\small I}\coagd{#1}}

\newcommand{\inter}{\normalfont\textrm{\small INTER}}

\newcommand{\emm}{\mathfrak{m}}

\DeclareMathOperator{\gammadist}{Gamma}
\newcommand{\iidsim}{\overset{\mathrm{i.i.d.}}{\sim}}
\newcommand{\dee}{\mathrm{d}}
\newcommand{\give}{{\hspace{0.08em}|\hspace{0.08em}}}

\def\[#1\]{\begin{equation}\begin{aligned}#1\end{aligned}\end{equation}}
\newcommand{\calI}{{\mathcal{I}}}
\newcommand{\calB}{{\mathcal{B}}}

\newcommand{\bbN}{\mathbb{N}}
\newcommand{\calN}{{\mathcal{N}}}
\newcommand{\bbP}{\mathbb{P}}
\newcommand{\1}[1]{\mathrm{1}_{\{#1\}}}
\newcommand{\mathwrap}[1]{\texorpdfstring{#1}{TEXT}}
\newcommand{\tB}{\widetilde{B}}

\usepackage[nottoc,numbib]{tocbibind}

\begin{document}

\title{\sc\bf\large\MakeUppercase{
Network and interaction models for data with hierarchical granularity
via
fragmentation and coagulation
}
}
\author{\sc  Lancelot~F.~James, Juho Lee, and 
Nathan Ross}
\date{\it Hong Kong University of Science and Technology, Korea Advanced Institute of Science and Technology, and University~of~Melbourne
}

\maketitle

\begin{abstract}
We introduce a nested family of Bayesian nonparametric models for network and interaction data 
with a hierarchical granularity structure that naturally arises through
finer and coarser population labelings. In the case of network data,
the structure is easily visualized by merging and shattering vertices, while respecting
the edge structure. We further develop Bayesian inference procedures for the model family, and apply
them to synthetic and real data. 
The family provides a connection of practical and theoretical interest between the Hollywood model of 
Crane and Dempsey, and the generalized-gamma graphex model of Caron and Fox. 
A key ingredient for the construction of the family 
 is fragmentation and coagulation duality for integer partitions,
and for this we develop novel duality relations that generalize those of Pitman and Dong, Goldschmidt and Martin.
The duality is also crucially used in  our inferential procedures. 
\end{abstract}


\tableofcontents

\section{Introduction}

This paper is concerned with the statistical analysis of interaction or association
data, which consists of a collection of subgroups of a population that have interacted in some way. For example, a corpus of scientific papers has a population of authors, and 
each paper consists of a subset of those authors. The data in this case consists of the multi-set of these subsets of authors. 
For another example, a population of animals could be grouped into pairs according to their mating; such cases where all interactions are pairs can be represented as networks. 
See \cite[Section~1.1]{Dempsey2022} and references there for other examples of association data with some discussion of different perspectives.

In such data, it may be the case that individuals from the population can be partitioned by
grouping under some trait. In the scientific corpus data, each author has a main institutional affiliation, and in the mating data, each bird has a tree they were born in. Relabeling individuals according to their group leads to a new set of data that is coarser than the original data. There may be more levels of granularity, e.g., by relabeling institutions by country, or
trees by region of a forest. Going the other way, the data can be refined by additional information, such as including subject areas in the corpus data. At the finest level, all labels are distinct, and at the coarsest, all labels are the same. For a simple concrete example, consider a corpus of four papers of a population of authors labeled $A1, A2, A3, A4$ given by
$\clc{(A1, A2),  (A3, A4), (A4,A3), (A4,A1)}$ (here we assume author order is meaningful). Now assume $A1$ and $A2$ are at UC Berkeley in the USA, while $A3$ is at Penn State  in the USA, and $A4$ is at NUS in Singapore. Then for the population of institutions, the association data becomes  
$\clc{(UCB, UCB), (PSU, NUS), (NUS, PSU), (NUS,UCB)}$, and  it becomes $\clc{(USA, USA), (USA, SIN), (SIN, USA), (SIN, USA)}$  for the population of countries (note here the collection of data is a multiset). 
Inferential procedures will be improved by taking into account this more \emph{hierarchically granular} structure.

An important case of this hierarchical interaction data that we pay special attention to here,  
is when all associations are pairs (as in the example above).  In this case, the data can be viewed as a \emph{network} or \emph{graph} where the nodes are the individuals, (directed) edges represent the pairwise associations, and we allow for loops and multiple edges between vertices. 
Statistical network modelling is a large and developing area with implications
in a huge number of areas including biology and sociology; see e.g, \cite{Crane2021} for a recent overview. 
In this case, coarsening the data by relabeling individuals according to some trait corresponds to merging vertices in the graph while respecting edges, and the dual operation of refining the data with additional information corresponds to shattering vertices. See Figure~\ref{fig:cfvert} for a network representation of the three levels of granularity for the example above.

\begin{figure}[!ht]
    \centering
    \includegraphics[width=0.95\linewidth]{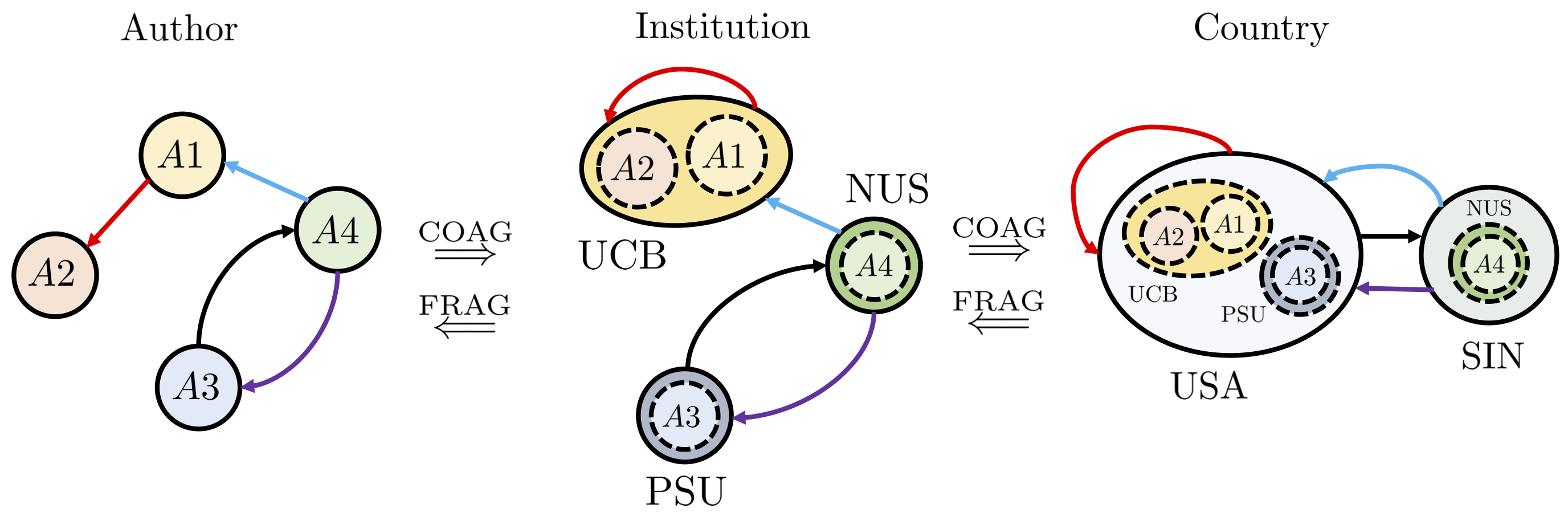}
    \caption{A visual network representation of levels of hierarchical granularity in the example from the text. 
    To go from left to right, vertices are merged or coagulated, while retaining all existing edges  (colored for clarity), and to go from right to left, the vertices are shattered or fragmented to create new vertices, with existing edges  distributed as per the example. In the institution and country level figures, the outer solid lines represent the vertex, and the internal dotted lines are for illustrative purposes only.}
    \label{fig:cfvert}
\end{figure}

The main purposes of this paper are to 1) introduce a Bayesian nonparametric family of probability distributions that model hierarchically granular associations, and which has important consistency properties that lead to a ``nesting'' within the family, 2) develop procedures for Bayesian inference of the parameters in the model family, and 3) apply these procedures to real data. A secondary purpose is to provide foundational perspective and tools for this kind of hierarchical data. The model family (see Definitions~\ref{def:icrpg} and~\ref{def:icrp}) builds off the edge-exchangeable ``Hollywood Model'' of \cite{Crane2018}, and the consistency properties stem from \emph{fragmentation} and \emph{coagulation} operations for set and interval partitions that we develop starting from those of \cite{Pitman1999b} and \cite{Dong2006}, the former of  which forms the basis of hierarchical Pitman-Yor processes. Our models have a similar flavor to the sequence memoizer introduced in \cite{Wood2009, Wood2011} for sequence data, and have the potential for analogous uses in association and network data.

In the network setting, our model family can be related to the 
 two parameter family of Caron-Fox (CF) random networks derived from \emph{generalized gamma} (GG) subordinators; see \cite[Section~6.3]{Caron2017}. 
This family, and the ``graphex'' generalizations of \cite{Veitch2015}, are popular statistical network models;
see \cite{Borgs2018}, \cite{Veitch2019}, \cite{Borgs2019b} and \cite{Borgs2021}. 
While the CF network derived from a gamma subordinator leads to a tractable model, it does not exhibit a power law degree distribution, which is a common feature of many real-world networks. However, our work shows the perhaps surprising result that appropriate fragmentation and coagulation operations applied to the CF gamma subordinator network model lead to a modified CF-GG network that behaves like a Pitman-Yor process, see \cite{Pitman1996} and \cite{Ishwaran2001}, which does exhibit power law behavior.
This allows us to establish an  explicit connection between the vertex-exchangeable CF-GG network model and the edge-exchangeable Hollywood model (Proposition~\ref{prop:cfpois1} and Theorem~\ref{thm:gcrp2cf}) through these operations. 
Furthermore, the fragmentation and coagulation operations create
a nesting of networks (see Definition~\ref{def:HICRP}), in parallel to Pitman-Yor and related processes where dual fragmentation and coagulation operations have been applied.
Our construction of the family 
relies on a
 generalization of the fragmentation/coagulation duality of \cite{Dong2006} for partitions, along with a new duality that combines that of \cite{Pitman1999b} and our generalization of \cite{Dong2006} (Theorem \ref{thm:pdgmmdual}), which is crucial for the inferential procedures
developed here. 
To our knowledge, this is the first instance where 
the duality of \cite{Dong2006} has been applied in Bayesian nonparametric settings.


The rest of the paper is organized as follows: To finish the introduction, we first briefly discuss some existing models that 
consider hierarchical structure in association or network data, all of which differ from our setting, 
and then introduce the model,  its motivations, and its connection to the CF-GG network.
 In Section~\ref{sec:genres} we develop a connection between integer partitions and association data, 
 and then state our new duality result,  which allows us to finally define our model family. In Section~\ref{sec:icrpinfproc}
 we develop inferential procedures for the model and in Section~\ref{sec:nums} apply them to synthetic data (as a check) and to a Wikipedia voting dataset.

\subsection{Related work}
Probably closest to our setting are the hierarchical models  introduced by  \cite{Clauset2008},  where  the 
vertices in the graph are also leaves in a binary tree. 
The latent hierarchical structure is 
encoded through genealogical closeness, 
and the interpretation of these hierarchies is not straightforward, depending 
in a complicated way on parameters of the model and the structure of the tree.  
 These models have been generalized to allow 
generic and random trees, and to the Bayesian setting; see \cite{Schmidt2013} and \cite{Peixoto2019}; and they are mainly used to infer latent community structure from a single network; see the introduction of \cite{Li2020} for a thorough discussion of this perspective. Our setting differs from these works in that we assume the hierarchical community structure respects the \emph{same} edges at each level, and each level can be viewed as a label category.  This added structure may decrease model flexibility, but increases interpretability.

There is a related but separate thread of research around ``multilayer'' networks, which gives a flexible modelling paradigm for a number of correlated networks. While multilayer models are flexible enough to encompass hierarchical structure as defined here, they are typically used as models for data of the form of a number of networks 
 with a number of shared nodes, e.g., a social network in a school observed over time.
 See \cite{Kivela2014} for a survey and discussion. For similar purposes, there is also the recent work \cite{Chandna2022}, which defines a very general class of models of correlated stochastic block models on the same set of vertices. 

The recent work of \cite{Dempsey2022} develops a hierarchical model for association data, also starting from the HM, but their notion of hierarchy is quite different from ours.

\subsection{Model definition, motivation, and connections}

We first formally define interaction data, roughly following \cite{Crane2018}.

\begin{definition}[Interaction set]
A finite (ordered) sequence from  a population with labels $\cL:=\{U_1,\ldots, U_k\}$,
is called an 
 \emph{interaction} of $\cL$. 
A multi-set $\{I_j\}_{j=1}^n$ of interactions of $\cL$ is called an interaction set of $\cL$.
\end{definition}

To digest this definition, consider the paper-corpus example from the introduction, but with the addition of two further papers that involve a fifth author $A5$:
\ben{\label{eq:ordint}
\cI:=\clc{(A1, A2),  (A3, A4), (A4,A3), (A4,A1), (A2,A5,A4), (A5, A2, A4, A3)}.
}
To consider this an interaction set, the labels are  $\cL:=\{A1, A2, A3, A4, A5\}$ (so $U_i = Ai$ and $k=5$)
and the interactions $\{I_j\}_{j=1}^6$ ($n=6$) are the elements of the set $\cI$, which are simply sequences of 
 elements from $\cL$.
Note that interactions are assumed to be ordered (in the example, the author-order of the papers is meaningful), and the interaction sets are unordered (in contrast to \cite{Crane2018} where ordered interaction sets are referred to as interaction processes), but our framework could easily be adapted to 
any choice of ordered/unordered at both levels. (The appropriate structure will depend on the data.) On a first reading of what follows,
it may be helpful to  consider the case where all interactions are pairs, and then the interaction set can be visualized as a multi-graph with vertices~$\cL$ and directed edges $\{I_j\}_{j=1}^n$, as described previously.

Our model extensively uses the \emph{Chinese Restaurant Process} (CRP); see \cite[Chapter~3]{Pitman2006}, which 
frequently appears in Bayesian nonparametric settings. 

\begin{definition}[CRP sequence]\label{def:crpseq}
Fix $\alpha\in[0,1)$,  $\theta\geq-\alpha$, or $\alpha=1$ and $\theta>-1$, and an atomless probability measure $F$ on $\IR$. 
Recursively define the sequence of random variables $(X_1,X_2,\ldots)$  as follows. 
Set $X_1\sim F$. For $n\geq1$, given $(X_1,\ldots, 
X_{n})$ consists of unique elements $\{U_1,\ldots, U_{K_n}\}$ and $N_n(i):=\babs{\{j : X_j = U_i\}}$, $X_{n+1}$ is a draw from the 
probability measure
\be{
\sum_{i=1}^{K_n} \frac{N_n(i) - \alpha}{n+\theta} \delta_{U_i} +   \frac{\theta+K_n\alpha }{n+\theta} F.
}
We write $(X_1,X_2,\ldots)\sim \CRP(\alpha,\theta; F)$ and $(X_1,\ldots, X_n) \sim \CRP_n(\alpha, \theta; F)$, and in the case that $ F=\mathrm{Uniform}(0,1)$, we write $ \CRP(\alpha,\theta)$ and $ \CRP_n(\alpha,\theta)$.
\end{definition}
\begin{remark}
It is helpful for later to note the two extreme cases
\ben{\label{eq:crpex}
(X_1,X_2,\ldots)=
\begin{cases}
(U_1, U_2, U_3, \ldots), & \alpha=1, \theta>-1, \\
(U_1, U_1, U_1, \ldots), & \alpha\in[0,1), \theta=-\alpha.
\end{cases}
} 
\end{remark}

With this definition, we can define in our language
a version of the \emph{Hollywood Model} (HM) from \cite[Section~4]{Crane2018}; see also \cite[Example~7.8]{Janson2018} and \cite[Section~4.2]{Bloem-Reddy2017}.
\begin{definition}[Hollywood Model]\label{def:HM}
Fix $(\alpha, \theta)$ in the allowable range of Definition~\ref{def:crpseq} and an atomless probability  measure $ F$ on $\IR$.
Let $(X_n)_{n\geq1}\sim\CRP(\alpha,\theta; F)$ be an infinite CRP sequence, independent of 
 $(L_i)_{i\geq1}$, which are i.i.d.\ with distribution $p:=(p_q)_{q\geq1}$ supported on $\IN$. Writing $\overline{ L}_j:=\sum_{i=1}^j L_i$, 
we say the sequence of interaction sets $(\cI_n)_{n\geq1}$  
is distributed as the Hollywood Model (HM) with parameters $(\alpha, \theta; F, p)$, denoted $(\cI_n)_{n\geq1}\sim\HM(\alpha,\theta; F, p)$ 
and $\cI_n \sim\HM_n(\alpha, \theta; F,p)$, if 
they have the joint representation 
\be{
\cI_n=\bclc{(X_{\overline{L}_{i-1}+1},\ldots, X_{\overline{L}_i})}_{i=1}^n.
}
We drop $ F$ from the notation when $ F=\mathrm{Uniform}(0,1)$.
\end{definition}

To explain the definition, the probability  distribution $p$ governs the i.i.d.\ lengths $L_1,L_2,\ldots$ of interactions, 
and $\cI_n$ consists of grouping the sequence $X_1, X_2,\ldots,$ into blocks of length $L_1,\ldots, L_n$.
For the network case $p_2=1$, each block is a pair and 
\be{
\cI_n = \{(X_i, X_{i+1})\}_{i=1}^n
}
is a random directed multi-graph with $n$ edges and vertices with labels $(U_1,\ldots, U_{K_{2n}})$, where $(U_i)_{i\geq1}$ are i.i.d.~$ F$ and $K_{2n}$ is as in Definition~\ref{def:crpseq}. 
The properties of the network associated to $\cI_n$ can be read from properties of the $\CRP$. For example, it is easily seen to be edge-exchangeable in the sense of \cite{Crane2018}; see also \cite{Cai2016} and \cite{Janson2018}. (It is also vertex exchangeable since the labels are i.i.d.~$ F$ (non-atomic), although this is somewhat artificial since it is the same as labelling the $n$ vertices according to  a uniform permutation of $[n]$, which always yields a vertex exchangeable graph.) 
Detailed results on the number of vertices and degree distribution 
are derived using this correspondence; see \cite[Theorems~4.2 and 4.3]{Crane2018}.

\begin{remark}
The i.i.d.\ vertex labeling by  an atomless $ F$ serves as a mathematical contrivance to give unique labels to each vertex, but is also useful to make distributional statements.
  In inference procedures, the labels given to each vertex will be 
human readable, e.g., considering the  example in the introduction, the labels
are the author name, institution, or country.
\end{remark}

With these ingredients, we can define our model, which is simply the HM with a random number of interactions.

\begin{definition}[CRP Interaction set]\label{def:icrpg}
Fix $(\alpha, \theta)$ in the allowable range of Definition~\ref{def:crpseq}, an atomless probability  measure $ F$ on $\IR$, a 
probability distribution $p:=(p_q)_{q\geq1}$ supported on $\IN$, and a random variable $N\in\IN$.
Letting 
 $(\cI_n)_{n\geq1}\sim\HM(\alpha,\theta; F, p)$ independent of $N$, we define the 
define the  $\CRP$ interaction set distribution with parameters $(\alpha, \theta;  F, p)$ and $N$
by 
\ben{\label{eq:ICRPdefg}
\ICRP_N\clr{\alpha, \theta;  F, p}:= \law(\cI_N).
}
We drop $ F$ from the notation when $ F=\mathrm{Uniform}(0,1)$.
\end{definition}

The definition only depends on $\law(N)$, but it is notationally simpler to 
parameterize by the random variable. Allowing  $N$ random increases the model flexibility, but 
is not so easy to infer. We highlight a special choice of $N$
that connects the $\ICRP$ to the CF-GG network model, which will be used in our inference procedures below.

\begin{definition}[ICRP special case]\label{def:icrp}
Retaining the notation from Definition~\ref{def:icrpg}, and letting $\nu:=(\nu_q)_{q\geq1}$ be
 a sequence of non-negative integers
with generating function $\phi_\nu(t):=\sum_{q\geq1} \nu_q t^q<\infty$ for all $t\geq0$,  $\clr{\gamma_s}_{s\geq0}$ be a standard gamma subordinator (an independent increments process with marginals $\gamma_s$ having density $x^{s-1} e^{-x} dx$ for $x>0$) independent of $\clr{N(s)}_{s\geq0}$, a rate one homogenous Poisson process on $\IR_+$. 
We define the special 
  $\CRP$ interaction set distribution with parameters $(\alpha, \theta;  F, \nu)$
by 
\ben{\label{eq:ICRPdef}
\SICRP\clr{\alpha, \theta;  F,\nu}:=\ICRP_{N\clr{\phi_\nu(\gamma_\theta)}}\bclr{\alpha, \theta;  F, \bclr{\nu_q \gamma_\theta^q /\phi_\nu(\gamma_\theta)}_{q\geq1}}.
}
We drop $ F$ from the notation when $ F=\mathrm{Uniform}(0,1)$.
\end{definition}

\begin{remark}
In contrast to the $\HM$ Definition~\ref{def:HM}, we allow $\nu$ to be a generic (finite) measure on $\IN$. 
On the right hand side of~\eq{eq:ICRPdef}, the distribution governing interaction length  
is invariant to scaling of $\nu$ (since $(\nu_q/\phi_\nu)_{q\geq1}$ is invariant to the scaling), while the number of interactions  $N\clr{\phi_\nu(\gamma_\theta)}$ is not. Thus we can control 
the total number of interactions without affecting the proportion of each type, which is important for model flexibility and inference. 
\end{remark}

To understand the motivation behind the $\SICRP$ definition, we consider the case $\nu= \lambda \delta_2$ ($\nu_2=\lambda>0$, and $\nu_i=0$ for $i\not=2$.), in which case
\ben{\label{eq:icrpgraphcf}
\{(X_i,X_{i+1})\}_{i=1}^{N(\lambda \gamma_\theta^2)}\sim \SICRP(\alpha, \theta,  F,\lambda \delta_2),
}
with $(X_i)_{i\geq1}\sim \CRP(\alpha, \theta; F)$. 
While this representation may not seem any less mysterious, it is in fact closely related to the Caron-Fox Generalized Gamma (CF-GG) network defined in
\cite[Section~6.3]{Caron2017}. 

\begin{definition}[Caron-Fox Generalized Gamma network]\label{def:cfgg}
For $(\sigma, \tau)\in [0,1) \times (0,\infty)$, denote the generalized gamma L\'evy density  by
\ben{\label{eq:gglm}
\rho_{\sigma, \tau}(y):= \frac{1}{\Gamma(1-\sigma)} y^{-1-\sigma} e^{-\tau y}, \,\, y>0,
}
and let $\Theta=\Theta_{\sigma,\tau}:= (U_i, W_i)_{i\in \IN} \subseteq \IR_+ \times \IR_+$ be the points of a 
Poisson point process with intensity $du \, \rho(y)  dy$.
Given $\Theta$, 
 let $(N_{ij})_{i,j\in\IN}$ be conditionally independent with $N_{ij}\sim\Po(W_i W_j)$. 
We define the family of networks $(G_{\sigma,\tau}(\theta))_{\theta \geq0}$ 
by the multi-edge sets
\be{
E(G_{\sigma,\tau}(\theta))=\cup_{i,j\in \IN: (U_i, U_j)\in[0,\theta]^2} \{(U_i , U_j)\}_{k=1}^{N_{ij}},
}
where we interpret $\{\cdot\cdot\}_{j=1}^0 = \emptyset$ so that $G_{\rho}(\theta)$ has no isolated vertices,
and the vertex set consists of the labels appearing in the edge set. 
We say the graph $G_{\sigma,\tau}(\theta )$ has the  \emph{Caron-Fox Generalized Gamma} distribution with parameters~$(\sigma,\tau)$ and $\theta$, denoted $G_{\sigma,\tau}(\theta)\sim \CF_{\sigma,\tau}(\theta)$.
\end{definition}
In words, to construct $G_{\sigma,\tau}(\theta)$, we start with vertices labelled $U_i\leq \theta$, and independently put $N_{ij}$ directed edges
from $U_i$ to $U_j$. If a vertex has no edges attached to it, then it is removed. 
Now, when $\sigma=0$, the weights $(W_i)_{i\geq1}$ can be understood as the jumps of a gamma subordinator; i.e., the process $\tau^{-1}\gamma_\theta:=\sum_{i} W_i\II[U_i\leq \theta ]$ which  has independent increments and marginal distributions $\Gam(\theta, \tau)$ with density proportional to $\tau^{\theta} y^{\theta-1} e^{-\tau y} dy$ for $y>0$.  Relatedly, in the case that $\alpha=0$, we can represent the $\CRP$ sequence as sampling from a probability distribution with masses equal to the normalized jumps of a gamma subordinator as described in, e.g., \cite[Section~4.2]{Pitman2006}, and from the beta-gamma algebra, we can represent these normalized jumps to time $\theta$ by  $W_i\II[U_i\leq \theta ]/\gamma_\theta$, which are jointly independent of $\gamma_\theta$. Using this representation in~\eq{eq:icrpgraphcf}, Poisson thinning implies the following connection between the CF-GG and $\SICRP$ models, which is also
implicitly shown in \cite[Section~6.3]{Caron2017}.

\begin{proposition}\label{prop:cfpois1}
For $\tau>0$ and $\theta>0$, we have
\be{
\SICRP(0,\theta; \mathrm{Uniform}(0,\theta), \tau^{-1} \delta_2) =  \CF_{0,\tau}(\theta).
}
\end{proposition}

Proposition~\ref{prop:cfpois1} connects the $\SICRP$ with the CF-GG model, and shows it is a natural
candidate for a tractable model of network data when $\nu_i=0$ for $i\not=2$. As previously mentioned, the one parameter model does not fit real-world data well, but both the CF-GG and the $\SICRP$ have two parameter versions that provide better fits. To compare these models, we first note that the two parameter $\SICRP$ is extremely easy to simulate via $\gamma_\theta$, $N(\gamma_\theta^2)$ and $(X_i)_{i\geq1}$, and thus can be used 
in Bayesian inferential procedures (as we do  below).
In contrast, for $G_{\sigma,\tau}(\theta)\sim\CF_{\sigma, \tau}(\theta)$ with $\sigma>0$, the analog of the CRP urn rule, given by \cite[(32) and (37)]{Caron2017}, read from \cite{Pitman2003}, depends on the analog of $\gamma_\theta$, and both analogs are much more complicated to sample.
Theorem~\ref{thm:gcrp2cf} below  gives an explicit relation between the two parameter $\SICRP$ and CF-GG models, which also provides further evidence that the $\SICRP$ is more natural for modelling purposes than the CF-GG;  see Remark~\ref{rem:labswd}.

Turning now to higher order interactions, a natural analog of the CF-GG network model for interactions of length $q$, would have, given $\Theta$, the number of interactions of $(U_{i_1},\ldots, U_{i_q})$ be conditionally independent and distributed Poisson$(W_{i_1}\cdots W_{i_q})$, which corresponds in the $\ICRP$ notation to
\be{
\{(X_{(i-1)q+1},X_{(i-1)q+2},\ldots, X_{iq})\}_{i=1}^{ N( \nu_q\gamma_\theta^q)}\sim \SICRP(\alpha, \theta,  F,\nu_q \delta_q),
}
for some weights $\nu_q$. 
If $\cI \sim \SICRP(\alpha, \theta,  F, \nu)$, then Poisson thinning implies that the 
the collection $\cI_q$ of interactions of $\cI$ of length exactly $q$ is distributed 
as $\SICRP(\alpha, \theta,  F,\nu_q \delta_q)$. The $\cI_q$ are dependent, since the labels 
are derived from the same sampling scheme, which is a desirable property. 
Altogether, the $\ICRP$ is a natural generalization of the one parameter CF-GG model which is easy to simulate, and, as we show in the next section, has desirable 
consistency properties for inference of interaction data with hierarchical granularity.

\section{Partitions and interaction data}\label{sec:genres}

To define fragmentation and coagulation operations on 
interaction sets, we first develop a correspondence between
interaction sets and partitions with some additional structure. The correspondence then allows us to easily transfer the
well-developed concepts of coagulation and fragmentation for partitions to interaction sets.
We need some notation for partitions in line with \cite{Pitman2006}. Here and below we write $\abs{\cdot}$ to denote cardinality of a set.
A partition $\pi=\{A_1, A_2,\ldots\}=\{A_i\}_i$ is a collection of disjoint subsets of $\IN=:\{1,2,\ldots\}$ called \emph{blocks}. If $\cup_i A_i = B$, then we may say that $\pi$ is a partition of $B$. 
We follow the convention that blocks are listed in order of least element and $A_k=\emptyset$ for $k$ larger than $\abs{\pi}$, the number of blocks of $\pi$. Similarly, for a partition $\pi$, we may write $\{A_1, \ldots, A_{\abs{\pi}}\}$ for $\{A_1,\ldots, A_{\abs{\pi}},\emptyset,\ldots\}$.
For any sequence $(X_1,X_2,\ldots)$, we associate a partition $\samp(X_1,X_2,\ldots)$ of~$\IN$ induced from this sequence by the equivalence relation given by $i\stackrel{r}{\sim} j$ if $X_i=X_j$ (i.e., putting indices with 
the same value in the same block). If $(X_1,X_2,\ldots)\sim \CRP(\alpha, \theta)$, then we also write $\samp(X_1,X_2,\ldots)\sim \CRP(\alpha,\theta)$ and $\samp(X_1,X_2,\ldots, X_n)\sim \CRP_n(\alpha,\theta)$, which should not cause confusion.
In addition, we define a \emph{composition} of a positive integer $m$ to be a sequence of positive integers summing to $m$.

\begin{definition}[Partition and composition of an interaction]\label{def:partint}
 Let $\cI$ be an interaction set of $\{U_1,\ldots, U_k\}$ with $\abs{\cI}=n$ interactions. 
Let $(I_j)_{j=1}^n$ be a uniform random ordering of the interactions of $\cI$ and write $\ell_j:=\abs{I_j}$ and $m=\sum_{j=1}^n \ell_j$. Associate to $\cI$ with this ordering the sequence 
$(X_i)_{i=1}^m:=(I_1,\ldots, I_n)$
of the labels from $(I_j)_{j=1}^n$ (slightly abusing notation). 
We define the partition of $\cI$ to be 
the partition induced by this sequence, $\partit(\cI):=\samp(X_1,\ldots, X_m)$.
We also define the \emph{composition} of $\cI$ by $\compos(\cI):=(\ell_1,\ldots, \ell_n)$.
 \end{definition}

 \begin{remark}
If interactions are unordered, then Definition~\ref{def:partint} should be modified to have an additional randomization of the order of each interaction.
\end{remark}

In the example given  at~\eq{eq:ordint} above, assuming random order of the interactions is reverse of how they appear in~\eq{eq:ordint}, we have $k=5$,  $n=6$, $m=15$, and 
\be{
(X_i)_{i=1}^{15}=(A5, A2, A4, A3, A2,A5,A4,  A4,A1,  A4,A3, A3, A4,   A1, A2).
}
Thus 
\ben{\label{eq:partex}
\partit(\cI)=\bclc{\{1, 6\}, \{2, 5, 15\}, \{3, 7, 8, 10, 13\}, \{4, 11, 12\}, \{9, 14\}},
}
and $\compos(\cI) = (4,3, 2, 2, 2,2)$.
As described previously, in the
 case where $\cI$  has all interactions of length two, we can visualize $\cI$ 
as a multi-graph and generate $\partit(\cI)$ as follows. Labeling the edges of the graph by a random permutation of $[n]$ and the edge-ends of edge $j$ by $(2j-1, 2j)$, each block of  
 $\partit(\cI)$  corresponds to the  the edge-end labels in a vertex. In this case, $\compos(\cI)=(2,\ldots, 2)$. 
 
Definition~\ref{def:partint} is useful because once we have a partition structure associated to an object, we have available the many tools and deep theory of random partitions, e.g., see \cite{Pitman2006}. However, this will only be helpful if it is possible to get back to the original interaction set structure. It is clear that this is not entirely possible since the labels in the population have been lost. However, the partition tells us how to construct a sequence that mimics the sequence $(X_i)_{i=1}^m$ in the sense that the positions of the distinct labels align, and the
composition tells us how to choose the lengths of interactions. 
The next definition gives a construction achieving this partial inversion.

 \begin{definition}[Interaction of a partition with composition]\label{def:intpop}
Let $\pi=\{A_1, \ldots, A_k\}$ be a partition of $[m]$ and $\cC=(\ell_1,\ldots, \ell_n)$ be a composition of $m= \bar{\ell}_n:=\sum_{j=1}^n \ell_j$. Let $(U_i)_{i\geq1}$
be a sequence of i.i.d.~$\mathrm{Uniform}(0,1)$ random variables and define the 
sequence $X'=(X_1', \ldots, X_{m}')$ by the following rule: if $j \in A_{i}$, then $X_j'=U_i$. 
We define the uniform-labelled \emph{interaction} of $(\pi, \cC)$  to be the  (random) interaction 
$\inter(\pi,\cC)=\bclc{(X_1',\ldots, X_{\ell_1}'),\ldots, (X'_{m_{n-1}+1}, \ldots, X_{m}')}$, 
where $\bar{\ell}_j=\sum_{i=1}^j \ell_i$.
\end{definition}

The Uniform$(0,1)$ labels ensure each block of the partition receives a distinct label. Continuing with the example above, if $\pi$ is given by~\eq{eq:partex} and $\cC=(4,3,2,2,2)$ and $U_1,\ldots, U_5$ are the distinct labels, then 
\be{
(X_i')_{i=1}^{15}=(U_1, U_2, U_3, U_4, U_2,U_1,U_3, U_3,U_5,  U_3,U_4, U_4, U_3,   U_5, U_2),
}
and so 
\be{
\inter(\pi,\cC)= \bclc{(U_1, U_2, U_3, U_4), ( U_2,U_1,U_3), (U_3,U_5),  (U_3,U_4), (U_4, U_3),  ( U_5, U_2)}.
}
Substituting $U_1= A5, U_2=A2, U_3=A4, U_4=A3, U_5=A1$ recovers~\eq{eq:ordint}. 
The randomized labels are  useful for distributional statements, but are not used in the inference, since in practice, the labels will have meaning. 
We could also change Definition~\ref{def:partint} to have embellished partitions, where the blocks are labeled by their population label.
In this case, the $U_i$ in Definition~\ref{def:intpop} would be the
relevant population label rather than random, and then the two definitions would provide a bijection between interaction sets and  compositions with partitions that have distinctly labeled blocks. 
In the case of networks, the correspondence between partitions and graphs is closely related to that of 
\cite[Section~2]{Bloem-Reddy2017}, where they view a network as an 
\emph{ordered} collection of edges, with elements of the population labeled by order of appearance in the sequence (more generally, this corresponds to assuming both the interactions and interaction sets are ordered, which is the simplest setting).

We record the following easy proposition which 
shows the way that it is possible to recover an intersection set from its partition and composition. 

\begin{proposition}\label{prop:intinv}
If 
$\cI$ is an interaction set with 
i.i.d.~$\mathrm{Uniform}(0,1)$ population labels, then 
 \be{
\inter\bclr{\partit(\cI),\compos(\cI)}\eqd \cI.
}
\end{proposition}

\subsection{Random interaction sets via random partitions}
The correspondence above allows us to lift constructions and properties of partitions with compositions to interactions. A first step in this direction is to define random interactions
through random partitions.  For exchangeable random partitions, this step has been taken in \cite{Crane2018}, albeit in slightly different language. 
Here we focus on partition sequences $(\Pi_n)_{n\geq1}$ generated according to the CRP defined in Definition~\ref{def:crpseq}. 
If $(X_1,X_2,\ldots)\sim \CRP(\alpha,\theta)$, then $\Pi_n \sim \samp(X_1,\ldots, X_n)$, jointly in $n$. 
Translating the rule for generating $(X_1,X_2,\ldots)$ of Definition~\ref{def:crpseq} to the sequence $(\Pi_n)_{n\geq1}$ gives the following well-known generative description of the $\CRP$ partition. Deterministically
$\Pi_1=\{1\}$, and given $\Pi_n=\{A_1,\ldots, A_{K_n}\}$,
$\Pi_{n+1}$ is generated by either adding $(n+1)$ to block $A_i$ with probability $(\abs{A_i}-\alpha)/(n+\theta)$, or   forming a singleton block of of $(n+1)$ with probability $(\alpha K_n + \theta)/(n+\theta)$. 
Note that $\alpha=1$ is deterministically equal to the partition of all singletons, and $-\theta=\alpha<1$  is deterministically equal to the partition with one block. 
We write $\Pi_n\sim\CRP_n\clr{\alpha,\theta}$ and $\Pi=(\Pi_n)_{n\geq1}\sim \CRP(\alpha,\theta)$.
With this definition, we have the following easy proposition that gives  a definition of the $\ICRP_N$ in terms of the $\CRP$ \emph{partition}.

\begin{proposition}\label{prop:part2ICRP}
Recalling the objects in the $\ICRP_N$ Definition~\ref{def:icrpg}:
 $(\alpha, \theta)$ is in the allowable range of Definition~\ref{def:crpseq};
$p:=(p_q)_{q\geq1}$ is a probability distribution, and $N\in\IN$ is a random variable.  
Letting $(L_i)_{i\geq1}$ be i.i.d.\ with distribution $p$ independent of $(\Pi_n)_{n\geq1}\sim \CRP(\alpha,\theta)$, and setting $\overline{L}_N:= \sum_{i=1}^N L_i$, we have 
\be{
\inter\bclr{\Pi_{\overline{L}_N}, (L_1,\ldots, L_N)} \sim \ICRP_N\clr{\alpha, \theta; p}.
}
\end{proposition}

\subsection{Fragmentation and coagulation for partitions}\label{sec:fragcoagpart}

As discussed in the introduction, we are interested in 
hierarchical interaction data, meaning we may refine the interaction set by adding information 
to the population labels, and coarsen it by
combining labels according to common traits. Since labels of an interaction set 
correspond to blocks of its partition, we can refine it by \emph{fragmenting} blocks of the partition, and coarsen it by \emph{coagulating} or merging sets of blocks. 
Using the correspondence between partitions 
and interaction sets developed in the previous section,
these operations correspond to merging and shattering labels in the interaction sets, while respecting
the interaction set structure; see  Figure~\ref{fig:cfvert} and discussion after Definition~\ref{def:intabfragcoag} below.

While there are many fragmentation and coagulation operations, e.g., see \cite{Bertoin2006}, we 
are interested in those that are closed with respect to the two parameter $\ICRP_N$ family, which, at the level of partitions, are those that are closed with respect to the $\CRP$. 
Our starting point are the operations defined in  \cite{Pitman1999b} and \cite{Dong2006}, which we 
refer to as Pitman  and DGM (respectively) fragmentation and coagulation. These operations are ``dual'' in the sense that, 
roughly speaking, for $0\leq\beta< \alpha\leq 1$ and $\theta>-\beta$, an appropriate Pitman coagulation turns a $\CRP(\alpha,\theta)$ partition  into a $\CRP(\beta, \theta)$ one, and a Pitman fragmentation  can take it back; see
 \cite[Theorem~12]{Pitman1999b} and also our forthcoming Theorem~\ref{thm:pdgmmdual} for details.
 For the other parameter, if $0\leq \alpha\leq 1$ and $\theta>-\alpha$, then DGM fragmentation
 turns a $\CRP(\alpha, \theta)$ into a $\CRP(\alpha, \theta+1)$, and a DGM coagulation takes it back; 
 see \cite[Theorem~3.1]{Dong2006} and Theorem~\ref{thm:pdgmmdual}. 
 Composing operations alters both parameters simultaneously, but it  is computationally and conceptually advantageous to combine
the sequential operation into one, which is one of our main contributions that  we now describe.

\medskip
\noindent{\bf Combined Pitman/DGM Fragmentation-Coagulation Duality.}
We first describe the general operations and then show duality.
For a set $A\subseteq N$ and partition $\pi$ of $\IN$, we 
write $\pi|_A$ for the partition of $A$ induced by restricting $\pi$ to $A$.

\begin{definition}\label{def:PDGMcoagfrag}
Let $0\leq  \beta\leq \alpha< 1$, $\theta\geq-\alpha$, and $\emm\in \IZ_{\geq0}$.

\noindent \emph{(PDGM $\emm$-fragmentation)}
Let $\pi=\{A_1,\ldots, A_k\}$ be a partition of $[n]$. 
We recursively define a sequence of random variables $J_1,\ldots, J_\emm$ taking values in $\{0,1,\ldots, k\}$ depending on $\pi$ according to $\CRP(\beta,\theta)$ probabilities started from $\pi$, meaning that for $1\leq \ell\leq \emm$, 
\be{
\IP(J_\ell = j | J_{1},\ldots, J_{\ell-1} ) = \frac{ \abs{A_j} + \abs{\clc{ 1\leq i< \ell: J_i = j}}-\beta}{n+\theta +\ell-1}, \,\,\, j=1,\ldots, k,
}
and
\be{
\IP(J_\ell =0 | J_{1},\ldots, J_{\ell-1} ) = \frac{\theta + k\beta+ \abs{\clc{ 1\leq i< \ell: J_i = 0}}}{n+\theta + \ell-1}.
} 
Given $(J_i)_{i=1}^\emm$, for $j=1,\ldots, k$, write $N_j = \abs{\clc{ 1\leq i\leq \emm: J_i = j}}$, and  let $\Pi_j\sim \CRP(\alpha,N_j-\beta)$ be conditionally independent partitions of $\IN$.
The $(\beta\to\alpha,\theta)$-PDGM $\emm$-fragmentation of~$\pi$ is the random partition of $[n]$ defined by
\be{
\fragd{\emm}_{\beta \to\alpha,\theta}(\pi):=\cup_{j=1}^k \{ \Pi_j |_{A_j} \}.
}

\noindent\emph{(PDGM $\emm$-coagulation)} 
Let $\{C_1, \ldots, C_{K_\emm}\} \sim \CRP_\emm(\beta, \theta)$ be a partition of $[\emm]$, and given $ \clc{C_1, \ldots, C_{K_\emm}}$, 
let
\ben{\label{eq:sdir}
S=(S_0, S_1, S_2, \ldots, S_{K_\emm})\sim \Dir\bbbclr{\frac{\theta + K_\emm\beta }{\alpha} ,\frac{|C_1|-\beta}{\alpha}, \ldots, \frac{|C_{K_\emm}|-\beta}{\alpha}}.
}
Given the above, let $M_1,M_2, \ldots$ conditionally i.i.d.\ taking values in $\{0,1,\ldots, K_\emm\}$ with 
\be{
 \IP(M_i=j)=S_j, \,\,\, j=0,1,\ldots, K_\emm.
}
Finally, let $\{B_1, B_2,\ldots\}\sim \CRP\bclr{\frac{\beta}{\alpha}, \frac{\theta+K_\emm \beta}{\alpha}}$ independent of the above and denote $\cB_0:= \{1\leq i\leq k: M_i=0\}$.
The $(\alpha\to\beta,\theta)$-PDGM $\emm$-coagulation of   $\pi:=\{A_1,\ldots, A_k\}$, a partition of $[n]$, is the random partition of $[n]$ defined by 
\be{
\coagd{\emm}_{\alpha\to\beta, \theta}(\pi):= \bclc{ \cup_{i:M_i =j} A_i}_{j=1}^{K_\emm}  \cup 
  \bclc{\cup_{i\in B_j\cap \cB_0}A_i }_{j\geq1}.
}
\end{definition}

In words, for the fragmentation, for each $j\in[k]$,  $A_j$ is conditionally independently fragmented according to $\CRP(\alpha, N_j-\beta)$.
For the coagulation operation, each block $A_j$ receives a conditionally independent label $M_j\in\{0,\ldots,K_\emm\}$, and blocks with common labels in $[K_m]$ are combined, and the blocks receiving the label zero are coagulated according to a $ \CRP\bclr{\frac{\beta}{\alpha}, \frac{\theta+K_\emm \beta}{\alpha}}$ applied to their indices. 

\begin{remark}[Relation to Pitman and DGM operations]\label{rem:relpdgmop}
If $\emm=0$, then in the PDGM fragmentation, we have $N_j=0$, and each block is fragmented by $\CRP(\alpha, -\beta)$. In the coagulation,  we have all $M_i=0$, and so all blocks are coagulated according to a $ \CRP\bclr{\frac{\beta}{\alpha}, \frac{\theta}{\alpha}}$ applied to their indices. This matches the operations described in \cite[Section~2.3]{Pitman1999b}.

If $\alpha=\beta$ and $\emm=1$, then the coagulation operation can be described by sampling an independent $P_0\sim\mathrm{Beta}((\theta+\alpha)/\alpha, (1-\alpha)/\alpha))$ variable, and then conditionally independently sampling a Bernoulli$(1-P_0)$ variable for each block,  merging those that receive a one, and doing nothing to the others;  here we use that $\{B_1,B_2,\ldots\}\sim \CRP(1,\frac{\theta}{\alpha}+1)$ consists of the partition of all singletons, from~\eq{eq:crpex}. For the same parameter range ($\alpha=\beta$ and $\emm=1$), the fragmentation operation selects a single block of $\pi$ in terms of the relevant CRP probability, and then fragments it by a $\CRP(\alpha,1-\alpha)$, and leaves the rest alone; noting that 
by~\eq{eq:crpex}, if $\Pi \sim \CRP(\alpha, -\alpha)$, then $\Pi|_A=A$. 
The analogous operation in \cite[Section~3]{Dong2006} is  defined on the masses of the latent Poisson-Dirichlet distribution
in the $\CRP$ construction, rather than on the blocks of the finite partition $\CRP$ partition. 
In the case of the fragmentation on the $\PD$ masses, the construction requires sampling a block randomly chosen according to mass, and fragmenting by PD$(\alpha, 1-\alpha)$, which depends only on the parameter 
$\alpha$. For the $\CRP$, the probability of choosing 
the latent $\PD$ mass corresponding to block $A_j$ is given by the chance 
that the next step in the $\CRP$ chooses that block. This occurs according to the $\CRP$ probabilities at the next step, governed by the random variable $J_1$. 
There is also a possibility that none of the masses of existing blocks are chosen, which corresponds to starting a new table  at the next $\CRP$ step, which we encode by $J_1=0$. Thus even for $\emm=1$ the operation has an interesting twist, and for $\emm>1$ it is new. 
\end{remark}

We have the following new duality result that generalizes 
\cite[Theorem~12]{Pitman1999b} and \cite[Theorem~3.1]{Dong2006}.

\begin{theorem}\label{thm:pdgmmdual}
Let $0\leq \beta\leq \alpha< 1$, $\theta\geq -\beta$, and $\emm\in\IZ_{\geq0}$. 
The following are equivalent:

\smallskip
\begin{enumerate}[nosep, label=(\roman*)]
\item\label{item:pdgmm1} $\Pi \sim\CRP_n(\beta,\theta)$ and $\Pi' \eqd \fragd{\emm}_{\beta\to\alpha, \theta}\clr{\Pi}$,
\item\label{item:pdgmm2} $\Pi'\sim\CRP_n(\alpha, \theta+m)$ and $\Pi\eqd \coagd{\emm}_{\alpha\to\beta,\theta}\clr{\Pi'}$.
\end{enumerate}
\end{theorem}

The  case where $\emm=0$ is exactly \cite[Theorem~12]{Pitman1999b},
and the case $\emm=1$ and $\alpha=\beta$ is an integer partition version of 
the interval partition result of \cite[Theorem~3.1]{Dong2006}. 
The general result can be understood from results in
\cite{James2010,James2013,James2015} and
\cite{Ho2021}, 
but we include a self-contained proof from a slightly different perspective in Section~\ref{sec:PDGMcoagfragpf} below.
At a high level, the proof of $(i)\implies(ii)$ essentially follows 
by showing that appropriate compositions of these simpler cases results in the fragmentation operation of 
Definition~\ref{def:PDGMcoagfrag}. The proof of $(ii)\implies(i)$ outside of the known cases 
relies on the subordinator description of the two-parameter Poisson-Dirichlet distribution given in \cite[Proposition~21]{Pitman1997a}, along with a size-bias description of the same given in
 \cite[Theorem~12]{Pitman1999b}, which 
 is a significant elaboration of the proof of  \cite[Theorem~3.1]{Dong2006}.
 See Section~\ref{sec:PDGMcoagfragpf} for the details.

\subsection{Hierarchically granular interaction sets via partitions}\label{sec:fragcoagint}
Using the correspondence between partitions and interaction sets of Definition~\ref{def:intpop} and Proposition~\ref{prop:intinv}
allows us to lift the coagulation and fragmentation concepts and results of Section~\ref{sec:fragcoagpart}
to interaction sets.

\begin{definition}\label{def:intabfragcoag}
Let $0\leq  \beta\leq \alpha< 1$, $\theta\geq-\alpha$, and $\emm\in \IZ_{\geq0}$.

\noindent\emph{(i)}
For an interaction set $\cI$, define
$(\beta\to\alpha,\theta)$-PDGM $\emm$-fragmentation of $\cI$ by
\be{
\ifragd{\emm}_{\beta \to\alpha,\theta}(\cI)=\inter\bclr{\fragd{\emm}_{\beta \to\alpha,\theta}\clr{\partit(\cI)},\compos(\cI)}.
}

\noindent\emph{(ii)}
For an interaction set $\cI$, define
$(\alpha\to\beta,\theta)$-PDGM $\emm$-coagulation of $\cI$ by
\be{
\icoagd{\emm}_{\alpha\to\beta, \theta}(\cI):=\inter\bclr{\coagd{\emm}_{\alpha\to\beta, \theta}\clr{\partit(\cI)},\compos(\cI)}.
}
\end{definition}

The notation is heavy, but the 
operations are natural: Considering the 
directed multi-network case where all interactions of $\cI$ have length $2$. 
As discussed just after Definition~\ref{def:partint}, 
after labeling edges $(2j-1, 2j)$,
 $\partit(\cI)$ has blocks 
that are the labels of edge-ends in a vertex. 
Then the fragmentation shatters vertices by 
appropriate $\CRP$ partitions, and each block 
represents a new vertex containing the edge ends of elements of the block, while respecting the edge structure
of the graph. 
Similarly,
to coagulate a graph, we
merge vertices according to appropriate $\CRP$ partitions. See Figure~\ref{fig:cfvert} for a visual representation.
The notion of shattering and merging labels applies in the general case, but it is difficult to visualize ``directed'' (ordered) hypergraphs with loops.

Using the correspondence between $\CRP$ integer partitions and $\ICRP$ interaction sets 
shown in Proposition~\ref{prop:part2ICRP}, we have the following 
fragmentation-coagulation duality lifted from the 
analogous duality of Theorem~\ref{thm:pdgmmdual} for partitions. 

\begin{theorem}\label{thm:intabdual}
Let $0\leq \beta\leq \alpha< 1$, $\theta\geq -\beta$,  $\emm\in\IZ_{\geq0}$, $N\in\IN$ a random variable
and $p=(p_i)_{i\geq1}$ a probability distribution on $\IN$. 
The following are equivalent:

\smallskip
\begin{enumerate}[nosep, label=(\roman*)]
\item\label{item:ipdgmm1} $\cI \sim\ICRP_N(\beta,\theta; p)$ and $\cI' \eqd \ifragd{\emm}_{\beta\to\alpha, \theta}\clr{\cI}$,
\item\label{item:ipdgmm2} $\cI'\sim\ICRP_N(\alpha, \theta+\emm;p)$ and $\cI\eqd \icoagd{\emm}_{\alpha\to\beta,\theta}\clr{\cI'}$.
\end{enumerate}
\end{theorem}

An important aspect of Theorem~\ref{thm:intabdual} is that
creates a ``nesting'' of interaction sets: If $0\leq \alpha_1 \leq \alpha_2< 1$ and $\emm\in\IN$, then
we can generate 
$\cI_1\sim\ICRP_N\clr{\alpha_1,\theta; p}$ from $\cI_2\clr{\alpha_2,\theta+\emm}\sim\ICRP_N\clr{\alpha_2,\theta+\emm;p}$,
by $\icoagd{\emm}_{\alpha_2\to\alpha_1,\theta}$, and this yields a joint distribution. We can also understand the joint distribution going the other way via fragmentation.
Thus, as discussed in the introduction, 
we can \emph{jointly} model an interaction set at $d$ levels (say) of increasing fineness, so that at the $j$th level the marginal distribution is $\ICRP_N(\alpha_j, \theta+\emm_j;p)$ with $\alpha_1\leq \alpha_2\leq \cdots\leq \alpha_d$ and $0=:\emm_1\leq \emm_2\leq \emm_3\leq \cdots \leq \emm_{d}$ with $\emm_i\in\IZ_{\geq0}$, as described in detail in the next definition.

\begin{definition}\label{def:HICRP}
Let $0\leq \alpha_1\leq \alpha_{2}\leq \cdots \leq \alpha_d<1$, $\theta >-\alpha_1$, $0=:\emm_1\leq \emm_2\leq \cdots \leq  \emm_{d}$ with $\emm_i\in\IZ_{\geq0}$ for $i=2,\ldots,d$, $N\in\IN$ a random variable, and 
$p=(p_i)_{i\geq1}$ a probability distribution supported on $\IN$.
We say the vector of interaction sets $(\cI_1,\ldots, \cI_d)$ is distributed as the \emph{Hierarchical ICRP} (HICRP) with parameters $(\alpha_i)_{i=1}^d, \theta, (\emm_i)_{i=2}^d, p$ and $N$, denoted $(\cI_1,\ldots, \cI_d)\sim\HICRP_N\bclr{(\alpha_i)_{i=1}^d, \theta, (\emm_i)_{i=2}^d; p}$, if $\cI_1\sim \ICRP_N(\alpha_1, \theta; p)$ and for $i=2,\ldots, d$,
we have
\be{
\law\bclr{\cI_i \, |\, \cI_{i-1},\ldots, \cI_{1} }=\law\bclr{ \ifragd{\emm_i-\emm_{i-1}}_{\alpha_{i-1} \to \alpha_{i},\theta+\emm_{i-1}}\clr{\cI_{i-1}}}.
}
Or, what is the same according to Theorem~\ref{thm:intabdual}, $\cI_d\sim \ICRP_N(\alpha_d, \theta+\emm_d;p)$ and for $i=1,\dots, d-1$
we have
\be{
\law\bclr{\cI_i \, |\, \cI_{i+1},\ldots, \cI_{d} }=\law\bclr{\icoagd{\emm_{i+1}-\emm_{i}}_{\alpha_{i+1}\to \alpha_{i}, \theta+\emm_{i+1}}\clr{\cI_{i+1}}}.
}
\end{definition}

Definition~\ref{def:HICRP} can be used as a model for hierarchically granular data of the kind described in the introduction. As a proof of concept, in the next section we describe a Bayesian inference procedure for the case $d=2$, and where $N$ is chosen so that the top $d=2$ level  is distributed   $\SICRP$.

\section{\mathwrap{Posterior inference for the $\HICRP$}}\label{sec:icrpinfproc}

We develop posterior inference for the $\HICRP_N\bclr{(\beta, \alpha), \theta, \emm; p}$,
where, matching the $\SICRP$ and using the notation of its Definition~\ref{def:icrp}, we set $p_q=\nu_q \gamma_{\theta+\emm}^q/\phi_\nu(\gamma_{\theta+\emm})$,
and $N= N(\phi_\nu(\gamma_{\theta+\emm}))$.
We first discuss inference for the marginal $\SICRP$.

\subsection{\mathwrap{Posterior inference for the $\SICRP$}}
Let $\calI_n$ be an observed interaction set with interactions $(I_i)_{i=1}^n$ with length $\ell_i := |I_i|$.
Let $Q \subseteq \bbN$ be the set of interaction lengths appearing in $\calI_n$, and $(n_q)_{q \in Q}$ be the number of interactions with length $q$, leading to $\sum_{q\in Q} n_q = n$ and $\sum_{q\in Q} qn_q = \sum_{i=1}^n \ell_i = \bar{\ell}_n$. Our goal is to fit $\mathrm{SICRP}(\alpha, \theta, \nu)$ to the data $\calI_n$. 
The generative process of $\calI_n$ from the SICRP model  from Definition~\ref{def:icrp} can be rewritten as below.
\[
&\gamma \sim \gammadist(\theta, 1), \\
& N \give \gamma = t \sim \mathrm{Poisson}(\phi_\nu(t)), \\
& L_1, \dots, L_n \give N=n, \gamma=t \iidsim \mathrm{Cat}\bbbclr{ \bbbclr{\frac{\nu_q t^q}{\phi_\nu(t)} }_{q \in Q} }, \\
&X_1, \dots, X_{\bar{\ell}_n} \give N=n, \bar{L}_n=\bar{\ell}_n \sim \CRP_{\bar{\ell}_n}(\alpha, \theta), \\
&I_i \give (\bar L_i)_{i=1}^n = (\bar\ell_i)_{i=1}^n := (X_{\bar{\ell}_{i-1}+1}, \dots, X_{\bar{\ell}_i}) \text{ for } i=1,\dots, n.
\]
Note that Proposition~\ref{prop:intinv} implies that it is enough to work with the joint likelihoods 
of $\partit(\cI_n)$ and $\compos(\cI_n)=(\ell_i)_{i=1}^n$, so let $(A_i)_{i=1}^{k} = \partit(\cI_n)$ be the partition induced by $\cI_n$. Each block $A_i$ corresponds to a label  from the population, and the number of interactions with the label corresponding to $i$ is the size of the block $|A_i|$.
Because the $\partit(\cI_n)$ is generating according to the exchangeable CPR, the joint density representing the generative process can be written as
\[
\lefteqn{f_\text{SICRP}(t, n, (\ell_j)_{j=1}^n, (A_i)_{i=1}^{k} \give \theta, \alpha, (\nu_q)_{q\in Q})} \\
&:= \frac{t^{\theta-1}e^{-t-\phi_\nu(t)}}{\Gamma(\theta)n!} \prod_{q\in Q} (\nu_q t^q)^{n_q} \times
\frac{\alpha^{k} \Gamma(\theta/\alpha+k)\Gamma(\theta)}{\Gamma(\theta/\alpha)\Gamma(\theta+n)}
\prod_{i=1}^{k} \frac{\Gamma(|A_i|-\alpha)}{\Gamma(1-\alpha)} \\
&\propto
\frac{t^{\theta-1}e^{-t-\phi_\nu(t)}}{n! \Gamma(\theta+n)} \prod_{q\in Q} (\nu_q t^q)^{n_q}
\times \frac{\alpha^{k}\Gamma(\theta/\alpha+k)}{\Gamma(\theta/\alpha)} \prod_{i=1}^{k} \frac{\Gamma(|A_i|-\alpha)}{\Gamma(1-\alpha)}.
\]

For posterior inference, we assume the prior densities for $(\alpha, \theta)$ as follows,
\[
&f_0(\alpha) = \mathrm{logit}\,\calN( \alpha \give 0, 1) \propto \frac{1}{\alpha(1-\alpha)} 
\exp\left( - \frac{(\log \frac{\alpha}{1-\alpha})^2}{2} \right), \\
&f_0(\theta) = \log\,\calN(\theta\give 0, 1) \propto \frac{1}{\theta}
\exp\left( - \frac{(\log \theta)^2}{2} \right).
\]
We find estimating $(\nu_q)_{q\in Q}$ along with $(t, \alpha, \theta)$ results in instabilities in inference. Instead we choose 
the prior for $\nu_q$ as,
\[
\bbP(\dee\nu_q\give \gamma=t) = \delta_{n_q/t^q}(\dee\nu_q) \text{ for } q \in Q,
\]
so that the identity $n_q = \nu_q t^q$ is always satisfied. The resulting target posterior density is written as,
\[
\pi_\text{SICRP}(t, \alpha, \theta, (\nu_q)_{q\in Q}) &\propto
\frac{t^{\theta - 1} e^{-t} }{n!\Gamma(\theta+n)}\prod_{q\in Q} \1{\nu_q=n_q/t^q} f_0(\alpha) f_0(\theta) \\
& \times \frac{\alpha^{k}\Gamma(\theta/\alpha+k)}{\Gamma(\theta/\alpha)}  \prod_{i=1}^{k} \frac{\Gamma(|A_i|-\alpha)}{\Gamma(1-\alpha)}.
\]
We can simulate this distribution via Metropolis-Hastings, where $t$ is 
drawn from $\gammadist(\theta, 1)$ and $(\alpha, \theta)$ are sampled with random-walk proposals.

\subsection{\mathwrap{Posterior Inference with PDGM $\emm$-Coag HICRP}}
Now assume we have two interaction sets, an original interaction set $\calI_n = \{ I_i \}_{i=1}^n$ and a coagulated interaction set $\calI'_n = \{ I_i'\}_{i=1}^n$. We want to fit the two-level HICRP model given at Definition~\ref{def:HICRP} to $(\calI_n, \calI_n')$. The model is given by the following generative process:
\[
\calI_N \sim \mathrm{SICRP}(\alpha, \theta+\emm, \nu), \quad
\calI'_N \give \calI_N \sim \mathrm{ICOAG}_{\alpha\to\beta}^{(\emm)}(\calI_N).
\]
We will see this situation has an additional challenge not present in the SICRP case. 
First recall the generative description of the PDGM $\emm$-coagulation of a given partition $\{A_i\}_{i=1}^k$, with slight changes in the notation,
\begin{equation}
    \begin{aligned}
        &\{C_1,\dots, C_{K_\emm}\} \sim \mathrm{CRP}_\emm(\beta, \theta), \\
& M_1, \dots, M_k \sim \mathrm{DirCat}\left( \frac{\theta + K_\emm\beta}{\alpha}, \frac{|C_1|-\beta}{\alpha}, \dots, \frac{|C_{K_\emm}|-\beta}{\alpha} \right), \\
&\mathcal{B}_j := \{ 1 \leq i \leq k : M_i = j \} \text{ for } j=0, 1, \dots, K_\emm, \\
&\{B_1, \dots, B_{K_{|\mathcal{B}_0|}}\} \give \mathcal{B}_0 \sim \mathrm{CRP}\left(\frac{\beta}{\alpha}, \frac{\theta + K_\emm\beta}{\alpha} \right)- {\mathcal{B}_0};
    \end{aligned}
\end{equation}
here $\mathrm{DirCat}$ denotes the Dirichlet-Categorical distribution and $\CRP(\alpha,\theta)-B$ denotes $\CRP(\alpha,\theta)$ applied to the elements of a set $B$. The coagulated partition is computed as
\[
\{A'_\iota\}_{\iota=1}^{K'} = \{ \cup_{i:M_i=j} A_i \}_{j=1}^{K_\emm} \cup \{ \cup_{i \in B_t} A_i \}_{t=1}^{K_{|\calB_0|}}.
\]
The joint probability of this process is computed as follows,
\besn{\label{eq:fcoaglhood}
\lefteqn{f_\text{COAG}(\{C_j\}_{j=1}^{K_\emm}, \{\calB_j\}_{j=1}^{K_\emm}, 
\{B_t\}_{t=1}^{K_{|\calB_0|}}\give \{A_i\}_{i=1}^k, \alpha, \beta, \theta, \emm)} \\
&= f_\text{CRP}(\{|C_j|\}_{j=1}^{K_\emm}\give \beta, \theta) 
f_\text{CRP}\left(\{|B_t|\}_{t=1}^{K_{|\calB_0|}}
\,\Big|\, \frac{\beta}{\alpha}, \frac{\theta + K_\emm\beta}{\alpha}
\right) \\
& \times 
\frac{\Gamma(\frac{\theta+\emm}{\alpha})}{\Gamma(\frac{\theta+\emm}{\alpha}+k)}
\frac{\Gamma(\frac{\theta+K_\emm\beta}{\alpha}+|\calB_0|)}{\Gamma(\frac{\theta+K_\emm\beta}{\alpha})}
\prod_{j=1}^{K_\emm}\frac{\Gamma(\frac{|C_j|-\beta}{\alpha} + |\calB_j|)}{\Gamma(\frac{|C_j|-\beta}{\alpha})},
}
where $f_\text{CRP}(\cdot\give \alpha, \theta)$ denotes the EPPF of $\CRP(\alpha, \theta)$. While this joint probability is easy to compute, we do not see all these variables in the observed interaction sets $(\cI_n, \cI'_n)$, and so they must be marginalized out, which presents computational difficulties. 
In particular, writing $K' = K_\emm + K_{|\calB_0|}$, we observe the collection
\[
\bbclr{\cup_{j=1}^{K_\emm}\{\cB_j\}} \cup\bbclr{\cup_{t=1}^{K_{\abs{\cB_0}}} \{B_t\}}=:\{\wt B_\iota\}_{\iota=1}^{K'}
\]
through the parent-children relationship between $\cI_n$ and $\cI_n'$, but not the partition $\{C_j\}_{j=1}^{K_\emm}$, $\abs{\cB_0}$, or whether or not a given block $\wt B_\iota$ originates from $\{\cB_j\}_{j=1}^{K_\emm}$ or not, and 
this information appears in the likelihood~\eq{eq:fcoaglhood}. 
Working with the observed $\cB_\iota$, we encode some additional information using the (unobserved) variables
\[
 Z_\iota = \left\{ \begin{array}{ll}
    1 & \text{ if } \tB_\iota \in \{ \calB_j \}_{j=1}^{K_\emm}, \vspace{0.2cm}\\
    0 & \text{ if } \tB_\iota \in \{ B_t \}_{t=1}^{K_{|\calB_0|}}.
\end{array}
\right.
\]
Ideally, we would compute the \emph{marginal} probability of the coagulation,
\[
\lefteqn{f_\text{MargCOAG}(\{\tB_\iota\}_{\iota=1}^{k'} \give \{A_i \}_{i=1}^k, \alpha, \beta, \theta, \emm)}\\
&= \sum_{C, Z} f_\text{COAG}(\{C_j\}_{j=1}^{K_\emm}, \{ \tB_\iota, Z_\iota \}_{\iota=1}^{k'} \give \{A_i\}_{i=1}^k, \alpha, \beta, \theta, \emm),
\]
but this requires  summing a the computationally prohibitive number. Instead, we introduce an unbiased estimator of $f_\text{MargCOAG}$ using a proposal distribution $q$, 
\[
\lefteqn{f_\text{MargCOAG}(\{\tB_\iota\}_{\iota=1}^{k'}\give \{A_i \}_{i=1}^k, \alpha, \beta, \theta, \emm)}\\
&=\mathbb{E}_{q}\left[ 
\frac{f_\text{COAG}(\{C_j\}_{j=1}^{K_\emm}, \{ \tB_\iota, Z_\iota \}_{\iota=1}^{k'} \give \{A_i\}_{i=1}^k, \alpha, \beta, \theta, \emm)}{q(\{C_j\}_{j=1}^{K_\emm}, \{Z_\iota\}_{\iota=1}^{k'})}\right] \\
&\approx 
\frac{1}{S}\sum_{s=1}^S \frac{f_\text{COAG}(\{C_j^{(s)}\}_{j=1}^{K^{(s)}_\emm}, \{ \tB_\iota, Z^{(s)}_\iota \}_{\iota=1}^{k'} \give \{A_i\}_{i=1}^k, \alpha, \beta, \theta, \emm)}{q(\{C^{(s)}_j\}_{j=1}^{K^{(s)}_\emm}, \{Z^{(s)}_\iota\}_{\iota=1}^{k'})},
\]
where
\[
\left\{\{C_j^{(s)}\}_{j=1}^{K_\emm^{(s)}}, \{Z_\iota^{(s)}\}_{\iota=1}^{k'} \right\}_{s=1}^S \iidsim
q.
\]
The proposal $q$ should be informative, so that the resulting estimator has a low variance. 
We design $q$ based on the PDGM $\emm$-fragmentation, the dual process of PDGM $\emm$-coagulation. Specifically, given $  \{A'_\iota \}_{\iota =1}^{k'} := \partit(\cI'_n) $, we generate the lost information as follows.
\begin{enumerate}
    \item For $1 \leq \ell \leq \emm$, draw $J_\ell \in \{ 0, 1, \dots, k'\}$ with probability,
    \[
    &\IP(J_\ell = \iota \give J_1, \dots, J_{\ell-1})
    = \frac{|A'_\iota| + |\{1 \leq \ell' < \ell : J_{\ell'}=\iota\}|-\beta}{\bar\ell_n + \theta + \ell-1}, \\
    &\IP(J_\ell = 0 \give J_1, \dots, J_{\ell-1})
    = \frac{\theta + k'\beta + |\{1 \leq \ell' < \ell : J_{\ell'}=0\}|-\beta}{\bar\ell_n + \theta + \ell-1}. \\
    \]
    \item Let $N_\iota := \{ 1 \leq \ell \leq \emm : J_\ell = \iota \}$ for $\iota = 0, 1, \dots, k'$. Then for $1 \leq \iota \leq k'$, compute
    \[
    Z_\iota = \mathbb{I}[N_\iota > 0], \quad \{C_1, \dots, C_{K_\emm}\} = \{ N_\iota : 1 \leq \iota \leq k', N_\iota > 0 \},
    \]
    with $\mathbb{I}[\cdot]$ denoting the indicator function.
\end{enumerate}
The probability of generating a proposal is simply computed as,
\[
q(\{C_j\}_{j=1}^{K_\emm}, \{Z_\iota\}_{\iota=1}^{k'}) = 
\frac{\Gamma(\theta + \bar\ell_n)}{\Gamma(\theta + \bar\ell_n + \emm)} \frac{\Gamma(\theta + k'\beta + |N_0|)}{\Gamma(\theta + k'\beta)}\prod_{\iota=1}^{k'} \frac{\Gamma(|A_\iota'| - \beta + |N_\iota|)}{\Gamma(|A_\iota'|-\beta)}.
\]
We estimate the posterior distribution of the parameters $(\alpha, \beta, \theta, \emm)$ with this unbiased estimator via the pseudo-marginal Metropolis-Hastings algorithm, see \cite{Andrieu2009}.
Considering the constraints, we reparameterize as follows:
\[
&(\alpha, \beta, \theta, \emm) \text{ where } 0 < \beta < \alpha < 1, \theta > 0, m \in \bbN\\
& \Rightarrow (\alpha, \zeta, \vartheta, \kappa) \text{ where } 
\beta = \zeta \alpha, \theta = \vartheta\kappa, \emm = \mathrm{ceil}(\vartheta(1-\kappa)) \\
& \quad\quad\text{ with } 0 < \zeta, \kappa < 1, \vartheta > 0.
\]
For $\alpha, \zeta, \kappa$, we place logit-normal priors $\mathrm{logit}\,\calN(0, 1)$, and for $\vartheta$, we place log-normal prior $\log\,\calN(0, 1)$.

\begin{remark}
We have developed an inference procedure for the case where there are two levels of granularity ($d=2$ in Definition~\ref{def:HICRP}). If there are more than two levels of granularity ($d>2$), then we can apply our two level procedure sequentially, by conditioning on lower levels to infer the parameters of the next. 
\end{remark}

\section{Numerical Study}\label{sec:nums}
In this section, we check our inference algorithm correctly 
recovers parameters in simulated data, and then apply it to Wikipedia voting data.

\subsection{Confirmation with synthetic interaction sets}
We generate synthetic data from the following PDGM Coag ICRP model:
\[
\cI_N \sim \mathrm{SICRP}(\alpha, \theta + \emm, \nu), \quad \cI'_N \sim \mathrm{ICOAG}_{\alpha\to\beta}^{(\emm)}(\calI_N),
\]
where
\ba{
&\alpha = 0.7, \,\, \beta = 0.2,\,\,\theta = 100, \,\, M = 10, \\
&Q = \{1, 2, 3, 7, 100, 1000\}, \\ 
&\widetilde{n}_1 = 500, \,\,\widetilde{n}_2 = 7000, \,\,\widetilde{n}_3 = 1000, \,\,\widetilde{n}_7 = 15, \widetilde{n}_{100} = 0.5, \,\,\widetilde{n}_{1000} = 1.
}
Here, $\widetilde{n}_q := \nu_q t^q$ is the expected number of interactions with length $q$. In other words, rather than directly specifying $\nu_q$, we implicitly choose it by specifying the expected number of interactions $\widetilde{n}_q$. The generated $\cI_N$ has 5,703 population labels and 8,418 interactions, and the coagulated $\cI'_N$ has 896 population labels and 8,476 interactions (note that the number of interactions does not change after the coagulation). Given $(\cI_N, \cI'_N)$, we run three independent MCMC chains, each comprising 30,000 steps. Posterior samples were collected after discarding the initial 15,000 burn-in samples, with a thinning interval of 10. To check the mixing of the MCMC chains, we report the~$\hat{R}$ diagnostics as defined in~\cite{Vehtari2021}; here $\hat{R}$ values  close to one indicate good mixing. Also, as a predictive check, we generate predictive interaction sets using the collected posterior samples of $(\alpha, \beta, \theta, \emm, (\nu_q)_{q\in Q})$ and compare the statistics of those predictive interaction sets to the observed interaction sets. The statistics we consider are ``degree'' distributions (the proportions of interactions participating in $k$ interactions, for $k\in\IN$), number of population labels in the the original interaction set, number of population labels in the coagulated interaction set, and the number of interactions. For the degree distribution, we compute the Kolmogorov-Sminorov (KS) statistic $D_\text{KS} := \sup_x | F_\text{obs}(x) - F_\text{pred}(x)| $, where $F_\text{obs}(x)$ is the CDF of the observed degree distribution and $F_\text{pred}(x)$ is the CDF of the predictive degree distribution. For the number of population labels and interactions, we report the Rooted Mean Squared Error (RMSE). Figure~\ref{fig:synthetic-joint} summarizes the results. The top row shows the coverage of the empirical distributions of the posterior samples around ground-truth values. The bottom row shows the degree distributions, number of population labels, and number of interactions of the predictive interaction sets generated from the estimated parameters. 
As a comparison, we also run two SICRP samplers \emph{independently} for the original interaction set $\cI_n$ and the coagulated set $\cI'_n$. Figure \ref{fig:synthetic-indep} summarizes the results, and Table~\ref{tab:synthetic} presents the quantitative comparison of the predictive distributions between the independent SICRPs and joint HICRP. While the joint model and independent models perform similar in general, the joint model better captures the distributions of the coagulated interaction set.

\begin{figure}[!ht]
    \centering
    \includegraphics[width=0.95\linewidth]{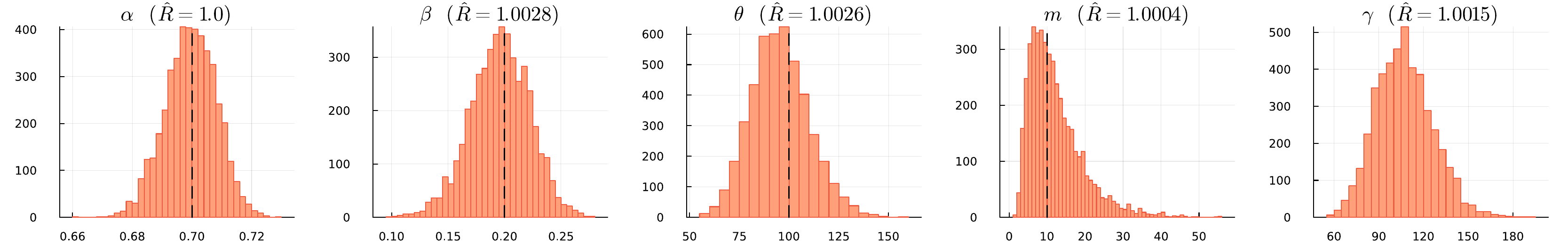}\vspace{0.2cm}
    \includegraphics[width=0.95\linewidth]{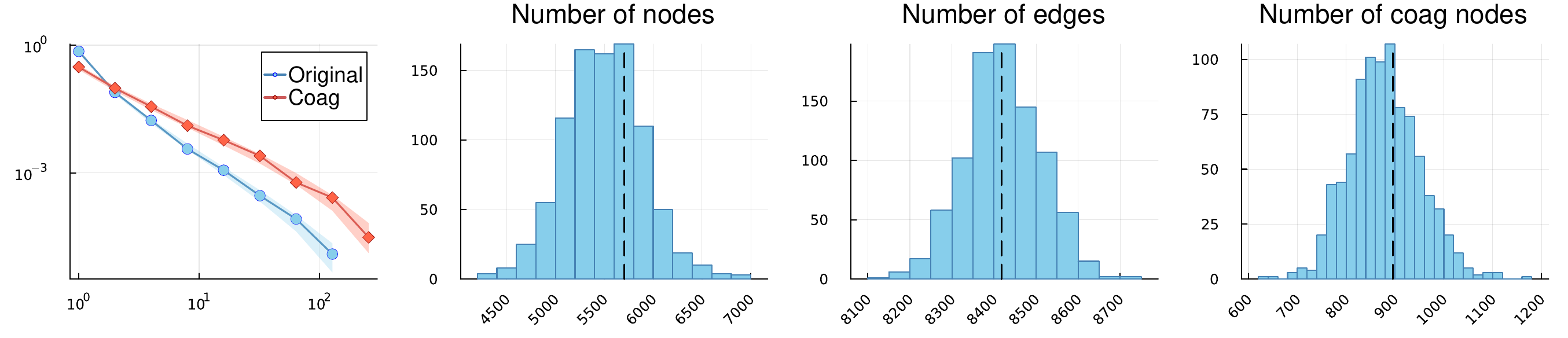}
    \caption{The results on synthetic interaction sets. The black dotted lines indicate the true values. Top: the estimated parameters with $\hat{R}$ diagnostics. Bottom: predictive statistics. For the degree distribution, the markers denote the observed degree distributions and the shaded area denotes the 95\% credible intervals of the predictive degree distributions.}
    \label{fig:synthetic-joint}
\end{figure}

\begin{figure}[!ht]
    \centering
    \includegraphics[width=0.48\linewidth]{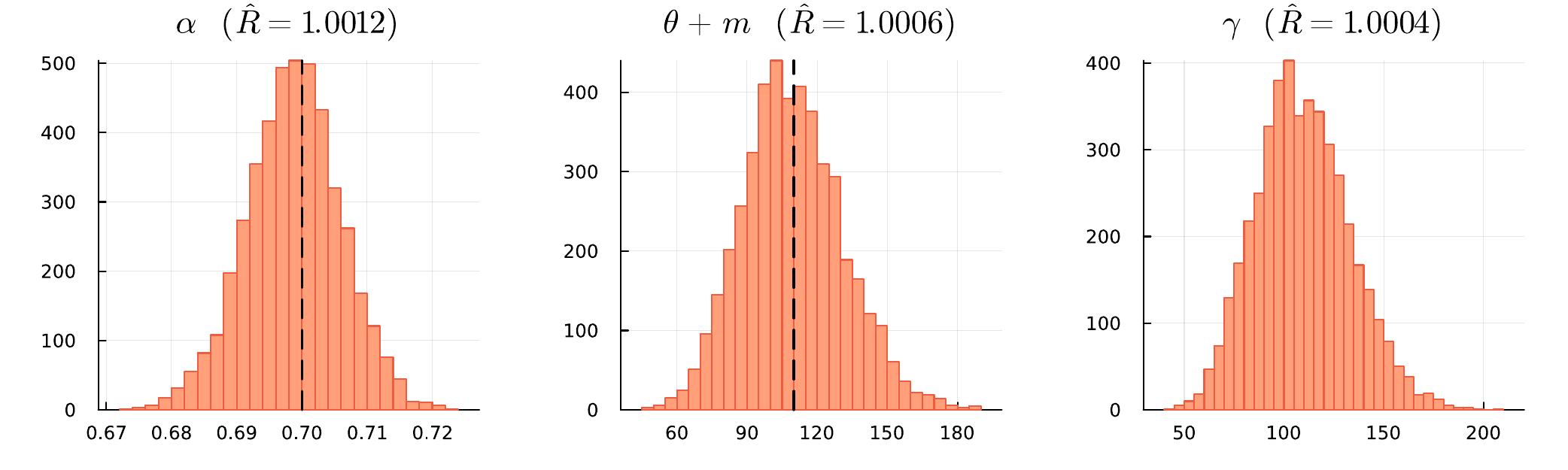}\vspace{0.2cm}
    \includegraphics[width=0.48\linewidth]{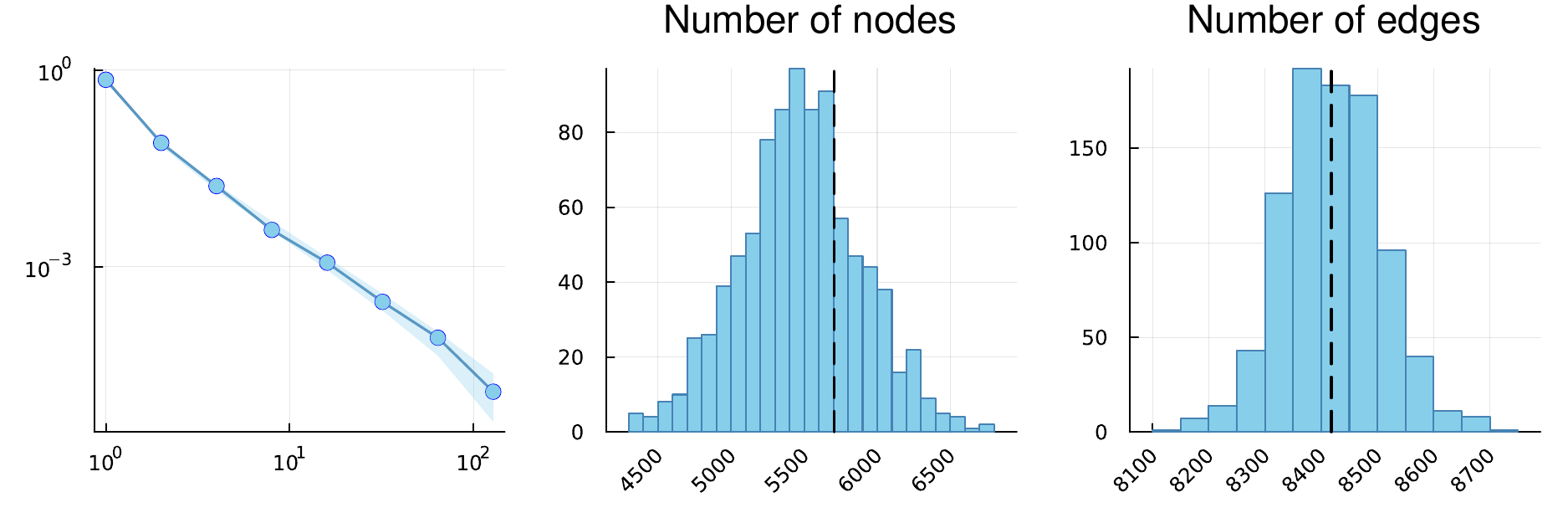}\vspace{0.2cm}
    \includegraphics[width=0.48\linewidth]{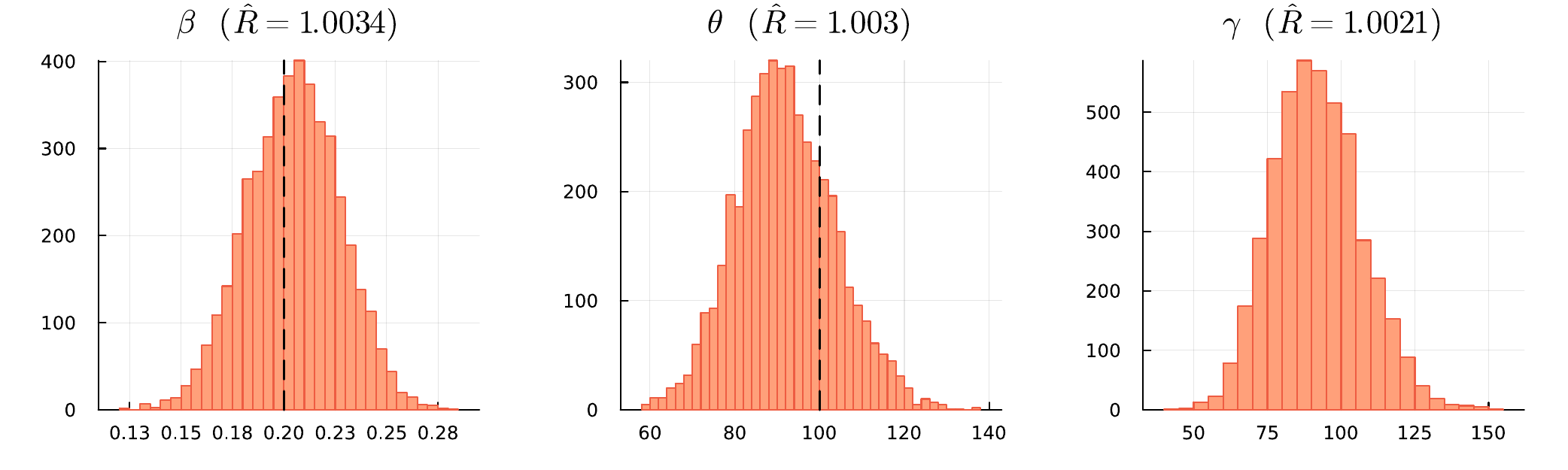}\vspace{0.2cm}
    \includegraphics[width=0.48\linewidth]{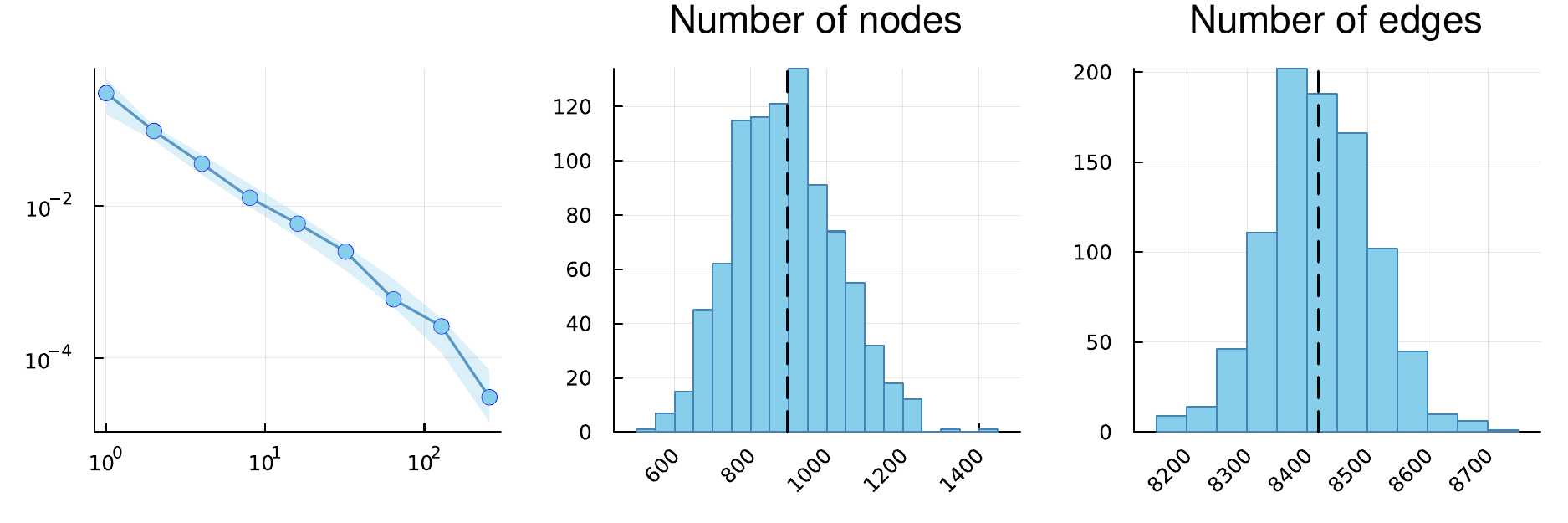}
    \caption{The results of independent SICRP samplers for the synthetic data. Top: the estimated parameters and predictive statistics for the original interaction set $\cI_n$. Bottom: the estimated parameters and predictive statistics for the coagulated interaction set $\cI'_n$. }
    \label{fig:synthetic-indep}
\end{figure}




\begin{table}[!ht]
    \centering
        \caption{Metrics for predictive statistics on synthetic data. SICRP and SICRP-COAG denote the independent SICRP run on the original interaction set $\cI_n$ and coagulated set $\cI'_n$, respectively. HICRP denotes the result with the PDGM $\emm$-coagulation HICRP run jointly on $(\cI_n, \cI_n')$.} 
        \small
    \begin{tabular}{cccccc}\toprule
                    & \makecell{\# labels\\RMSE} & \makecell{\# coag labels\\RMSE} & \makecell{\# interactions\\RMSE} & \makecell{degree dist\\$D_\mathrm{KS}$} & \makecell{coag degree dist \\$D_\mathrm{KS}$} \\\midrule
        SICRP & 474.5742 & - & 88.1626 & 0.011$\pm$0.0054 & - \\
        SICRP-COAG & -  & 135.1164  & 86.8058 & - & 0.0772$\pm$0.0453\\
        HICRP & 455.1763 & 90.7496 & 74.1134 &  0.0119$\pm$0.0060 & 0.0313$\pm$0.0117 \\\bottomrule
    \end{tabular}
    \label{tab:synthetic}
\end{table}

\subsection{Wikipedia election networks}

We further test our model on Wikipedia election network\footnote{\url{https://snap.stanford.edu/data/wiki-Elec.html}} from \cite{Leskovec2010}, which is a network consisting of English Wikipedia users voting for admin elections. The original graph contains 7,118 users (nodes) with 103,675 votes (edges). Each edge consists of a voter id, candidate id, and a voting timestamp. Given this graph, we create a coagulated graph as follows. We first collect all the timestamps and quantize them into 20,000 bins. A user (node) may vote or be voted multiple times, so there may be multiple timestamps associated with the user. Among those multiple timestamp values, we choose the latest timestamp as a representative timestamp of the user. Then, all nodes with their latest timestamp belonging to the same bin are merged into a single node. The resulting coagulated graph contains 4,308 nodes and 103,675 edges (the number of edges does not change).

Given the graphs, we run our PDGM $\emm$-coagulation HICRP sampler for the original and the coagulated graphs, and compare it to the SICRP samplers run independently for the original and coagulated graphs. Note here that the independent samplers are fundamentally limited, as they cannot capture the association between the nodes in two graphs, that is, the independent SICRPs cannot model the parent-children structures between the nodes in two graphs, while the joint model can naturally capture those structures and generate graphs with proper association between the nodes. We run three independent MCMC chains for 50,000 steps and collect posterior samples after 25,000 burn-in samples with a thinning interval of 10.

\begin{figure}[!ht]
    \centering
    \includegraphics[width=0.95\linewidth]{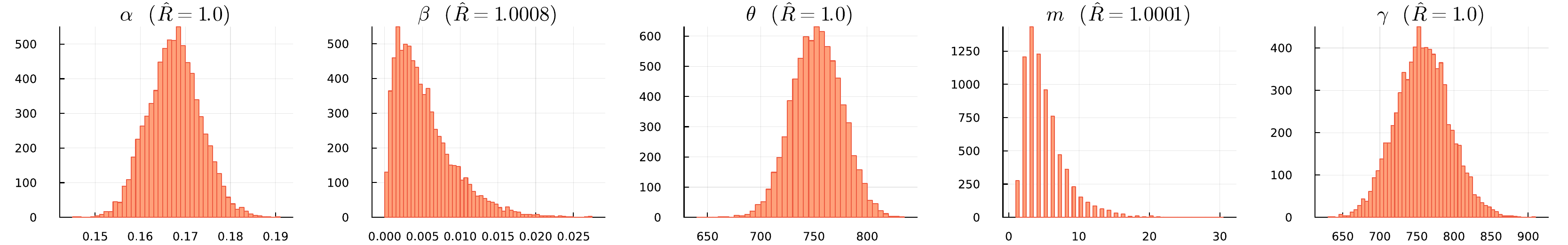}\vspace{0.2cm}
    \includegraphics[width=0.95\linewidth]{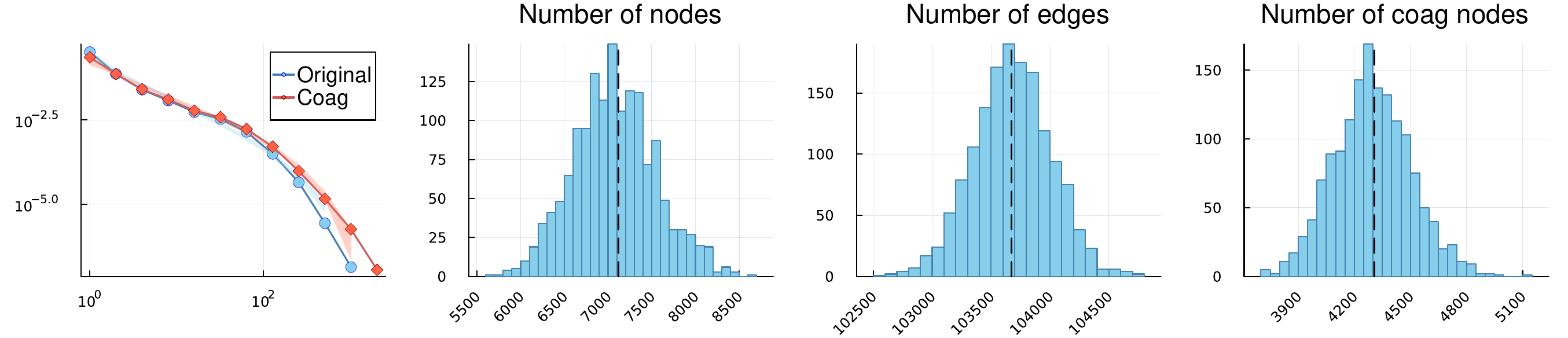}
    \caption{The results on Wikipedia election networks. Top: the estimated parameters with $\hat{R}$ diagnostics. Bottom: predictive statistics. For the degree distribution, the markers denote the observed degree distributions and the shaded area denotes the 95\% credible intervals of the predictive degree distributions.}
    \label{fig:wikivote-joint}
\end{figure}

\begin{figure}[!ht]
    \centering
    \includegraphics[width=0.48\linewidth]{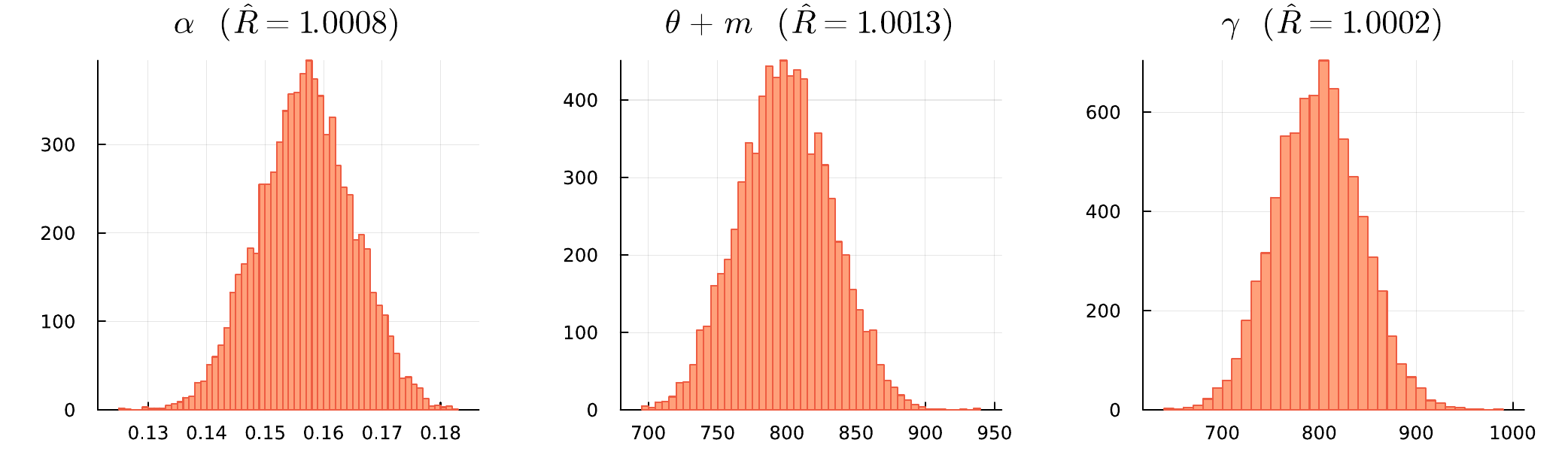}\vspace{0.2cm}
    \includegraphics[width=0.48\linewidth]{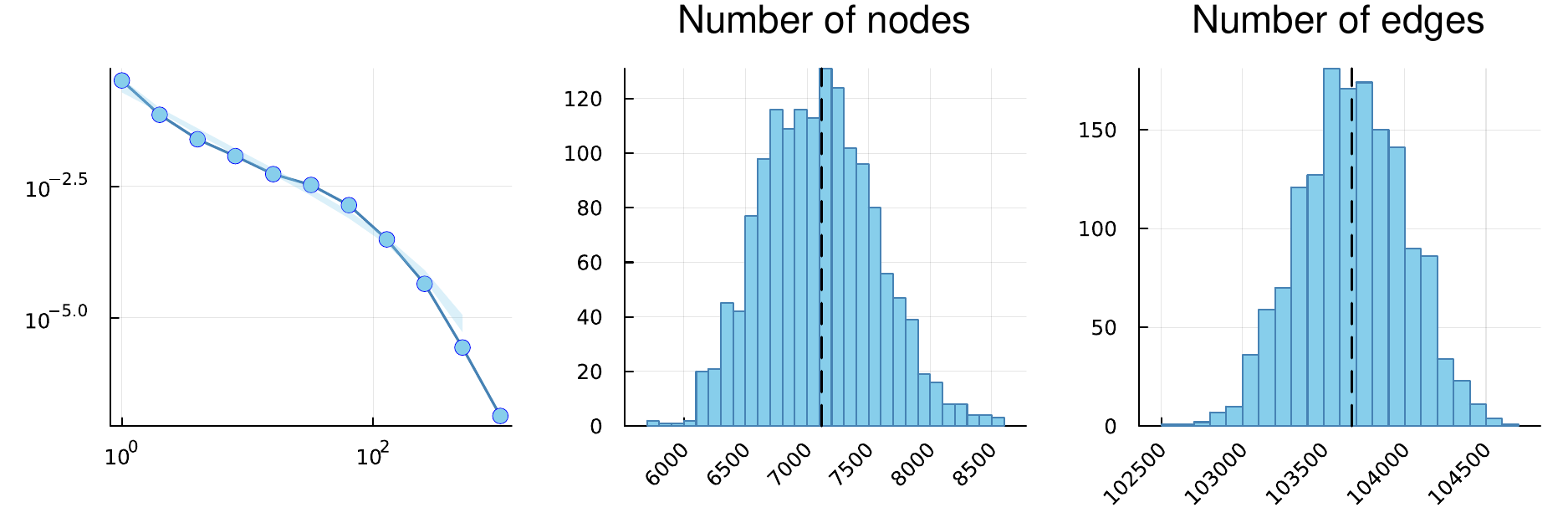}\vspace{0.2cm}
    \includegraphics[width=0.48\linewidth]{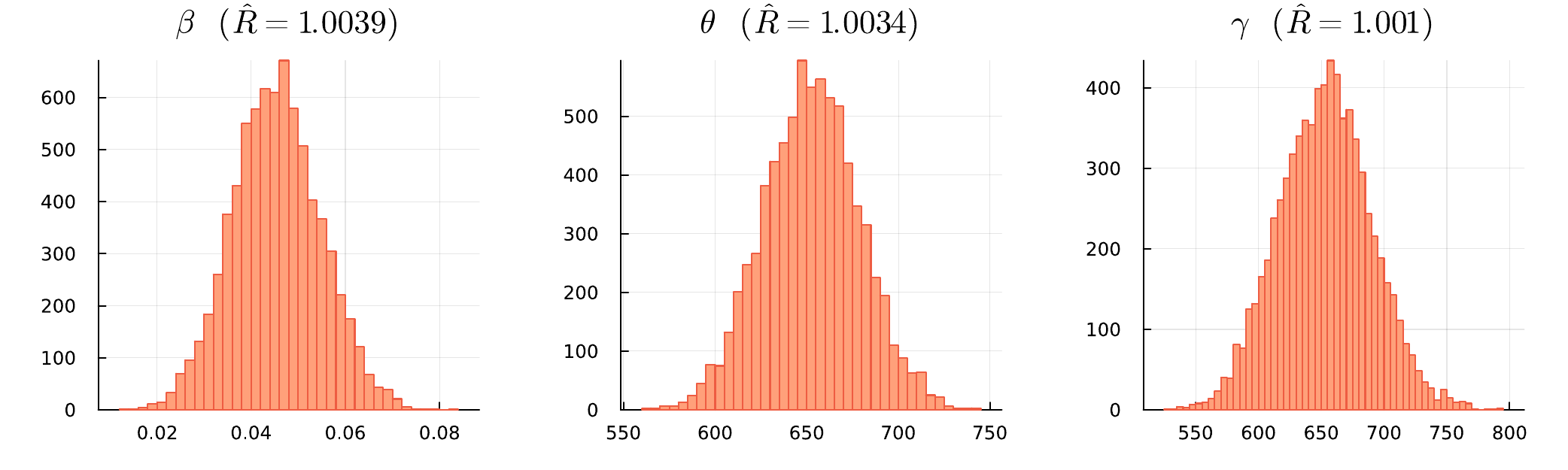}\vspace{0.2cm}
    \includegraphics[width=0.48\linewidth]{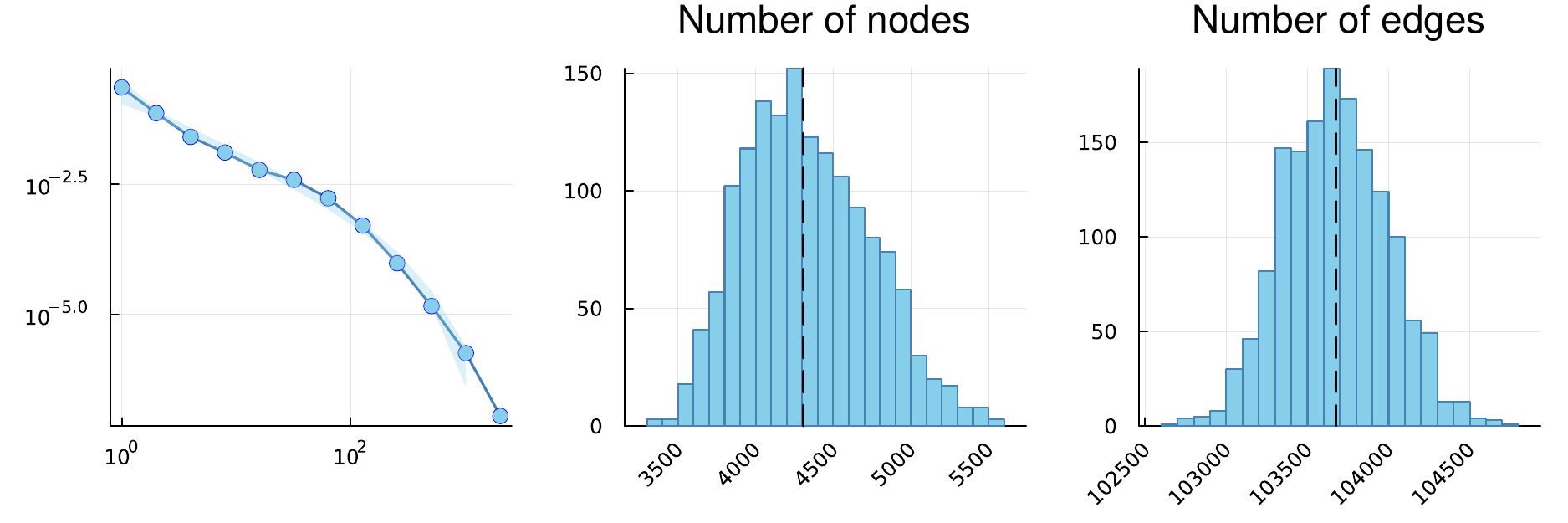}
    \caption{The results of independent SICRP samplers on Wikipedia election networks. Top: the estimated parameters and predictive statistics from the original Wikipedia election network. Bottom: the estimated parameters and predictive statistics from the coagulated network. }
    \label{fig:wikivote-indep}
\end{figure}

Figure~\ref{fig:wikivote-joint} shows the result from the HICRP sampler, and Figure~\ref{fig:wikivote-indep} shows the result with the independent SICRP samplers. Both joint model and independent models estimate similar parameter values, and generate reasonable predictive graphs. Table~\ref{tab:wikivote} compares the posterior predictive metrics. Similar to the synthetic data experiments,  the joint model and independent models perform similarly in general, but the joint model better captures the degree distributions of the coagulated network.

\begin{table}[!ht]
    \centering
        \caption{Metrics for predictive statistics on Wikipedia election network. SICRP and SICRP-COAG denote the independent SICRPs run on the original graph and coagulated graph, respectively. HICRP denotes the result with the PDGM $\emm$-coagulation HICRP run jointly on two graphs.} 
        \small
    \begin{tabular}{cccccc}\toprule
                    & \makecell{\# nodes\\RMSE} & \makecell{\# coag nodes\\RMSE} & \makecell{\# edges\\RMSE} & \makecell{degree dist\\$D_\mathrm{KS}$} & \makecell{coag degree dist \\$D_\mathrm{KS}$} \\\midrule
        SICRP & 458.6639 & - & 324.3876 & 0.0795$\pm$0.0203 & - \\
        SICRP-COAG & -  & 407.1574  & 321.7863 & - & 0.0745$\pm$0.0232\\
        HICRP & 474.1999 & 204.5967 & 327.4884 &  0.0771$\pm$0.0182 & 0.0560$\pm$0.0158 \\\bottomrule
    \end{tabular}
    \label{tab:wikivote}
\end{table}

To highlight that HICRP indeed captures the joint structure between the original graph and coagulated graph, we further compare two predictive statistics to observed graphs. Note that the independent SICRPs cannot compute the joint statistics. We first compare the number of children distributions of the predictive graphs to the actual number of children distribution, which is displayed in the left panel of Figure~\ref{fig:wikivote-joint_stats}. The model successfully simulates the number of children distribution. We also conducted a comparison of the average parent degrees, which we define as follows. In the original graphs, we categorized nodes based on their degrees using binning. Then, for each bin, we identified the parent nodes in the coagulated graphs and computed their average degrees. This analysis aims to observe the pattern of how the degrees of the original graph decrease in the coagulated graph following the coagulation operation. The right panel of Figure~\ref{fig:wikivote-joint_stats} compares the observed average parent degrees (depicted as blue dots) with the predictive parent degrees (illustrated as a shaded area representing the 95\% credible interval). Although there appears to be a disparity between the actual values, the predictive distribution adeptly captures the underlying pattern.

\begin{figure}[!ht]
    \centering
    \includegraphics[width=0.4\linewidth]{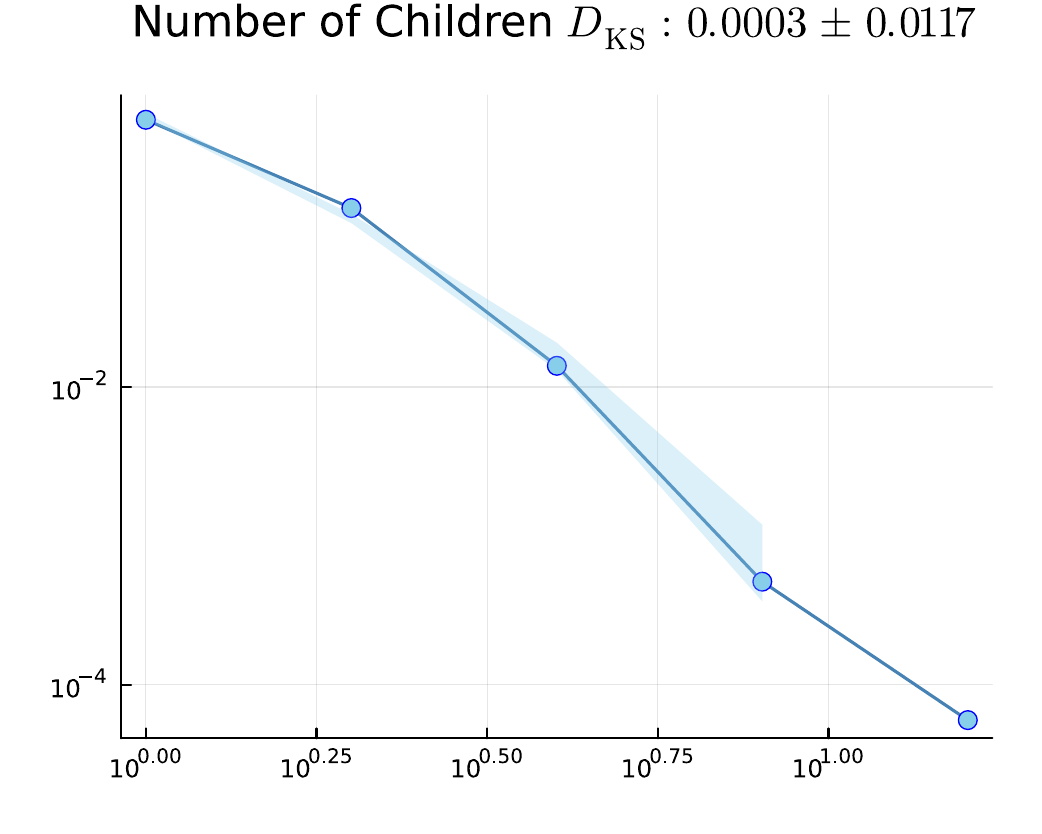}    \includegraphics[width=0.4\linewidth]{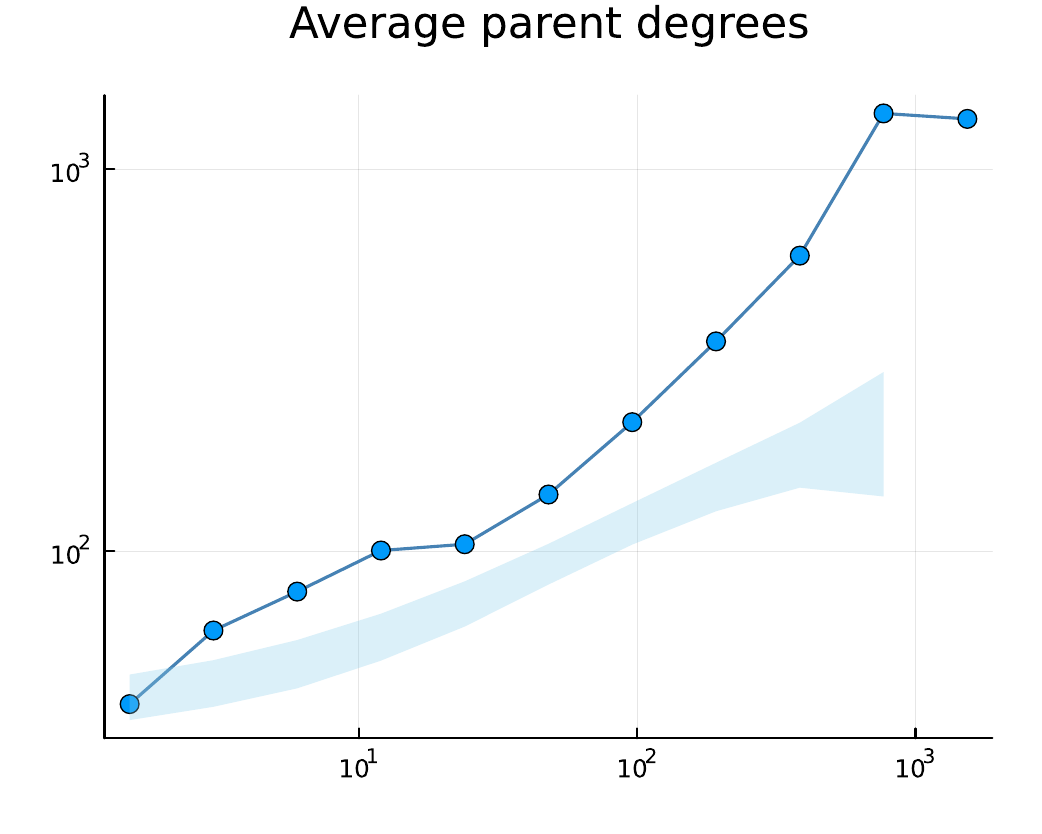}    
    \caption{Left: number of children distributions. Right: average parent degrees. The markers denote the observed distribution and the shaded area denotes the 95\% credible interval of the predictive distributions.}
    \label{fig:wikivote-joint_stats}
\end{figure}
\mbox{}

\section*{Acknowledgments}
LFJ was supported by grant RGC-GRF 16301521 of the Hong Kong SAR.
%
%


\section{Supplementary material: Background and proofs}

\subsection{\mathwrap{Relation of the two  parameter model to the  $\CF$-GG model}}
Proposition~\ref{prop:cfpois1} relates the one-parameter CF-GG model
to the $\ICRP$ by
\be{
\SICRP(0,\theta; \mathrm{Uniform}(0,\theta), \tau^{-1} \delta_2) =  \CF_{0,\tau}(\theta)
}
for $\tau>0$ and $\theta>0$. However, the one parameter model does not fit many real world networks well.
For example, for fixed $\tau$ and large $\theta$, the number of edges ($N(\tau^{-2}\gamma_\theta^2)$) is order $(\theta/\tau)^2$, and the number of vertices corresponds to the number of unique labels in the CRP, which is order $\theta \log\clr{2(\theta/\tau)^2 + \theta}$; see \cite[(3.13)]{Pitman2006}. 
Thus, for fixed $\tau$, the number of edges is essentially of order the number of vertices squared and the graph is dense. The two parameter CF model provides much more flexible behavior; see \cite{Caron2017} and \cite{Caron2023}, but is difficult to analyze. We have argued that our two parameter $\SICRP(\alpha,\theta;  \tau^{-1} \delta_2)$ is a simpler alternative, and the next result, 
 which follows in a straightforward way from  \cite[Proposition~21]{Pitman1997a}, shows how this $\SICRP$ and the two-parameter $\CF$ networks are related. We call the distribution of a random network without its vertex labels the \emph{unlabeled version}.

\begin{theorem}\label{thm:gcrp2cf}
Fix $0< \alpha < 1$, $\tau>0$ and let $(\gamma_s)_{s\geq1}$ be a standard gamma subordinator with L\'evy measure $\rho_{0,1}$ defined at~\eq{eq:gglm}. The unlabeled versions of
$
\CF_{\alpha, \tau}\clr{ \alpha \tau^{-\alpha} \gamma_{\theta/\alpha}}
$ 
and
$
\SICRP(\alpha, \theta; \tau^{-1} \delta_2)
$ have the same distribution. 
\end{theorem}

\begin{remark}
The unlabeled analog of Proposition~\ref{prop:cfpois1} is recovered by sending $\alpha\to0$,
noting that $\alpha \tau^{-\alpha} \gamma_{\theta/\alpha}\to \theta$ almost surely and in $L^2$.
\end{remark}

\begin{remark}\label{rem:labswd}
Considering the vertex-labels of $\CF_{\alpha, \tau}\bclr{ \alpha \tau^{-\alpha} \gamma_{\theta/\alpha}}$ and $\SICRP(\alpha, \theta; \tau^{-1} \delta_2)$ of Theorem~\ref{thm:gcrp2cf}, in the $\CF$ case these are uniform on the interval $(0, \alpha \tau^{-\alpha} \gamma_{\theta/\alpha})$, whereas in the $\SICRP$ case these are Uniform$(0,1)$. Because the number of vertices and edges in the CF-GG graph depends on 
the length of ``time'' considered in the subordinator, the labels will depend on the graph structure through $\gamma_{\theta/\alpha}$, which does not seem meaningful.
\end{remark}

\subsection{Proof of Theorem~\ref{thm:pdgmmdual}}\label{sec:PDGMcoagfragpf}

\begin{proof}[Proof of Theorem~\ref{thm:pdgmmdual}]
To show $(i) \implies (ii)$, we first consider the special cases where either $\emm=0$, or $\alpha=\beta$. 
When $\emm=0$, the result is exactly (ii)$\implies$(i) of \cite[Theorem~12]{Pitman1999b} with their $\beta$  equal to our $\beta/\alpha$; cf.\ Remark~\ref{rem:relpdgmop}. 
For $\beta=\alpha$, we show the result first for $\emm=1$, which is conceptually simpler. If  $\emm=1$, the statement is the integer partition analog 
of the interval partition result of \cite[Theorem~3.1]{Dong2006}; cf.\ Remark~\ref{rem:relpdgmop}; which can be connected through Kingman's ``paintbox'' construction
of the $\CRP$, see \cite[Chapter~3]{Pitman2006}, which we now describe. We first define the two-parameter Poisson-Dirichlet $\PD(\alpha, \theta)$ distribution, which can be realized as
 random variables $P_1 > P_2> \cdots>0$ with $\sum_{i\geq1}P_i=1$. The $P_i$ are the reordering of the random variables 
 $(B_i \prod_{j=1}^{i-1}(1-B_j))_{i=1}^\infty$ with independent $B_i  \sim \mathrm{Beta}(1-\alpha, \theta + i\alpha)$,  obtained via beta ``stick-breaking''. 
Letting $(U_i)_{i\geq1}$ and $(V_i)_{i\geq1}$ be i.i.d.\ sequences of $\mathrm{Uniform}(0,1)$ random variables independent of $(P_i)_{i\geq1}\sim \PD(\alpha,\theta)$, and setting $S_j:=\sum_{i=1}^j P_i$ and $S_0:=0$. Define 
\be{
X_i = \sum_{j\geq 1} U_j \II\bcls{S_{j-1} < V_i < S_j}. 
}
In words, the $X_i$ are conditionally i.i.d.\ with distribution putting mass $P_j$ at the value $U_j$.
The key facts about this construction are that the partition of $[n]$ given by $\Pi:=\samp(X_1,\ldots, X_n)\sim\CRP_n(\alpha,\theta)$, and, given $\samp(X_1,\ldots, X_n)$, the probability $X_{n+1} = X_i$ is $(\abs{\clc{1\leq j\leq n: X_j = X_i}}-\alpha)/(n+\theta)$; i.e., the $\CRP$ probability of joining the block associated to $X_i$; and the probability $X_{n+1}$ is a new, unseen $U_j$ value is
given by the remaining complementary proability, which is the $\CRP$ probability of starting a new block. 

Now, define the random variable $I$ by the identity  $X_{n+1} = U_{I}$, and let  $(Q_i)_{i\geq1}\sim \PD(\alpha,1-\alpha)$ independent of the above. By construction,  $\IP(I=i) = P_i$ marginally, and thus \cite[Theorem~3.1]{Dong2006} states that the concatenation of $(P_{I} Q_i)_{i\geq1}$ and $(P_j)_{j\not=I}$ put in decreasing order, 
is distributed as $\PD(\alpha, \theta+1)$.
Letting $(\widetilde U_i)_{i\geq1}$ be i.i.d.\ $\mathrm{Uniform}(0,1)$ independent of the variables above, and setting $T_j:=\sum_{i=1}^j Q_i$ with $T_0:=0$, define
\be{
\wt X_i = \sum_{j\not=I} U_j \II\bcls{S_{j-1} < V_i < S_j} + \sum_{j\geq1} \wt U_j \II\bcls{S_{I-1}+P_I T_{j-1} < V_i < S_{I-1}+P_I T_{j}}.
}
In words, the $\wt X_i$ are conditionally i.i.d.\ with distribution putting mass $P_j$ at the value $U_j$ for $j\not=I$,
and putting mass $P_IQ_j$ at the valued $\wt U_j$.
Thus, from \cite[Theorem~3.1]{Dong2006} and the paintbox construction, the partition of $[n]$ given by $\samp(\wt X_1,\ldots, \wt X_n)\sim\CRP_n(\alpha,\theta+1)$. 
Since we use the same $V_i$'s to generate both the $X_i$ and the $\wt X_i$, we also claim that 
\be{
\samp(\wt X_1,\ldots, \wt X_n)\eqd \fragd{1}_{\alpha\to\alpha, \theta}\bclr{\samp(X_1,\ldots, X_n)}.
}
This is because for any $X_i \not=U_{I}$, we have $X_i =\wt X_i$, and the blocks associated to those values are the same in the two samples. For the other indices (if any), the values from the sample satisfying
$\clc{1\leq j\leq n: X_j = U_I}$ are relabelled conditionally independently according to  $\sum_{i\geq1}  Q_i \delta_{ \wt U_i }$, and the partition generated from the sample is distributed as a $\CRP$ with parameters $(\alpha, 1-\alpha)$. This matches the description of the fragmentation operation $\fragd{1}_{\alpha\to\alpha, \theta}$ as long as the probabilities of choosing a block to fragment there are the same as the probabilities given by $J_1$. But this is the case, since, according to $\CRP$ probabilities, 
\be{
\IP\bclr{ X_{n+1} = X_i | ( X_1,\ldots, X_n)} = \frac{\abs{\clc{1\leq j\leq n: X_j = X_i}}-\alpha}{n+\theta}.
}

To cover the general-$\emm$ case, we first work at the level of random masses and generalize \cite[Theorem~3.1]{Dong2006}. Retaining the notation above, 
we set 
\be{
\samp\clr{X_{n+1},\ldots, X_{n+\emm}}=:\{C_1, \ldots, C_K \},
} 
and define $I_j$ via $U_{I_j}$ being the label associated to $C_j$. Given these random variables, let $(Q_i\sp{j})_{i\geq1}\sim \PD(\alpha, \abs{C_j}-\alpha)$ be conditionally independent for $j=1,\ldots, K$, and  independent of $(U_i\sp{j})_{i,j\geq1}$ which are i.i.d.\ $\mathrm{Uniform}(0,1)$. 
Then we claim that iterating \cite[Theorem~3.1]{Dong2006} implies that the concatenation of $(P_{I_j} Q_i\sp{j})_{i\geq1}$, $j=1,\ldots, K$ and $(P_j)_{j\not\in \{I_1,\ldots, I_K\}}$ put in decreasing order, 
is distributed as $\PD(\alpha, \theta+\emm)$.
This is because the fragmentation operation entails
choosing a mass at random, say $P_{I_1}$,  according to length, then fragmenting it by $\PD(\alpha, 1-\alpha)$. Iterating again,
we choose a mass according to length and fragment it by $\PD(\alpha, 1-\alpha)$. Now, the mass chosen is either a fragment of $P_{I_1}$ or not. In the first case, the normalized fragments of $P_{I_1}$ are distributed $\PD(\alpha, 1-\alpha)$, and, by \cite[Theorem~3.1]{Dong2006}, choosing one at random and fragmenting by $\PD(\alpha,1-\alpha)$ is distributionally the same as fragmenting $P_{I_1}$ by $\PD(\alpha, 2-\alpha)$. The probability of choosing a fragment of $P_{I_1}$ is the same as the second customer sitting at the same table as the first  customer in $\CRP(\alpha, \theta)$. Continuing in this way, we see that $\{C_1,\ldots, C_K\}$ encodes how many times each top-level mass is fragmented when iterating the procedure $\emm$ times, and if a mass is chosen $c$ times, it must be fragmented by $\PD(\alpha, c-\alpha)$.

With this result, we can move forward analogous to the case $\emm=1$, setting  $T_j\sp\ell:=\sum_{i=1}^j Q_i\sp\ell$ with $T_0\sp\ell:=0$, and defining
\be{
\wt X_i = \sum_{j\not\in \{I_1,\ldots, I_K\} } U_j \II\bcls{S_{j-1} < V_i < S_j} + \sum_{j=1}^K  \sum_{\ell\geq1}  \wt U_\ell\sp{j} \II\bcls{S_{I_j-1}+P_{I_j}  T_{\ell-1}\sp{j}  < V_i <S_{I_j-1}+P_{I_j}  T_{\ell}\sp{j}}.
}
From the $\emm$-generalization of \cite[Theorem~3.1]{Dong2006} derived above, we have that $\samp(\wt X_1,\ldots, \wt X_n) \sim \CRP_n(\alpha, \theta+\emm)$. 
Since we use the same $V_i$'s to generate both the $X_i$ and the $\wt X_i$, it is also the case that 
\be{
\samp(\wt X_1,\ldots, \wt X_n)\eqd \fragd{\emm}_{\alpha, \theta}\bclr{\samp(X_1,\ldots, X_n)}.
}
This is because for any $X_i \not\in \{U_{I_1},\ldots, U_{I_K}\}$, we have $X_i =\wt X_i$, and the blocks associated to those values are the same in the two samples. For the other indices (if any), the values from the sample satisfying
$\clc{1\leq k\leq n: X_k = U_{I_j}}$ are relabelled conditionally independently according to  $\sum_{i\geq1}  Q_i\sp{j} \delta_{ \wt U_i\sp{j} }$, and the partition generated from the sample is distributed as a $\CRP$ with parameters $(\alpha, \abs{C_j}-\alpha)$. This matches the description of the fragmentation operation $\fragd{\emm}_{\alpha, \theta}$ as long as the probabilities of choosing  blocks to fragment match those associated to $(J_1,\ldots, J_\emm)$, but this is the case from $\CRP$ probabilities, since
\be{
\IP\bclr{ X_{n+\ell} = X_i |(X_1,\ldots, X_{n+\ell-1})} = \frac{\abs{\clc{1\leq j\leq n+\ell-1: X_j = X_i}}-\alpha}{n+\theta+\ell-1},
}
which recursively match those of $\law(J_\ell | J_1,\ldots, J_{\ell-1})$.

Now, for the general case, using the $\emm=0$ and $\alpha=\beta$ cases just derived, 
if $\Pi\sim\CRP_n(\beta, \theta)$, then
\be{
\fragd{0}_{\beta\to\alpha, \theta+\emm}\circ\fragd{\emm}_{\beta\to\beta,\theta}(\Pi)\sim \PD(\alpha, \theta+\emm).
}
To see the composition is the same as $\fragd{\emm}_{\beta \to\alpha, \theta}$, note that the operation
  $\fragd{\emm}_{\beta\to\beta,\theta}$  conditionally independently fragments the block $A_j$ of $\Pi$ 
by a $\CRP(\beta, N_j - \beta)$, and then the operation $\frag_{\beta \to\alpha, \theta+\emm}$ fragments each of those blocks by independent $\CRP(\alpha, -\beta)$ partitions. But according to the case $\emm=0$ already shown,  fragmenting a $\CRP(\beta, N_j - \beta)$ by a $\CRP(\alpha,-\beta)$, is the same as fragmenting once by a $\CRP(\alpha, N_j -\beta)$, which is exactly the $\fragd{\emm}_{\alpha\to \beta, \theta}$ operation.

For $(ii)\implies (i)$, the case that $\emm=0$ is exactly (i)$\implies$(ii) of \cite[Theorem~12]{Pitman1999b} with their $\beta$  equal to our $\beta/\alpha$; cf.\ Remark~\ref{rem:relpdgmop}. 
Though not needed for our proof, note also 
the case $\beta=\alpha$ and $\emm=1$ follows from \cite[Theorem~3.1]{Dong2006}, along with Kingman's ``paintbox'' construction
of the $\CRP$ in a much simpler way than the fragmentation operation: Each block of the partition of $[n]$ 
corresponds to a mass in a $\PD(\alpha, \theta+1)$ interval partition, and  merging of masses in the interval partition  can be mapped in a one-to-one way to its corresponding integer partition block; cf.\ Remark~\ref{rem:relpdgmop}.
Similarly, for the general $\emm$ case, it's enough to show the result at the level of masses, which we do now.

For our proof in the general case, a key fact we use is \cite[Corollary~20]{Pitman1996}: Let $(\wt P_i)_{i\geq1} \sim \PD(\beta, \theta+K_\emm\beta)$, independent of $\{C_1, \ldots,C_{K_\emm}\} \sim \CRP_\emm(\beta, \theta)$, the random partition of $[\emm]$ defined in the coagulation description of Definition~\ref{def:PDGMcoagfrag}. Given $\clc{C_1, \ldots, C_{K_\emm}}$,  
let
\be{
T=(T_0, T_1, T_2, \ldots, T_{K_\emm})\sim \Dir\bclr{\theta + K_\emm\beta, |C_1|-\beta, \ldots, |C_{K_\emm}|-\beta }.
}
Then the ranking in decreasing order of the concatenation of $(T_{i})_{i=1}^{K_\emm}$ and $\clr{T_0 \wt P_i}_{i\geq 1}$  is distributed as $\PD(\beta, \theta)$. 

To see how this helps, note that our coagulation operation takes $(P_i)_{i\geq1}\sim \PD(\alpha, \theta +\emm)$ to the reranking of the concatenation of 
$
\bclr{\sum_{i: M_i = j}P_i}_{j=1}^{K_\emm}$ and $(\sum_{i: M_i =0, i \in B_j} P_i)_{j\geq1}$, where the $M_i$ are conditionally i.i.d., taking values according to the Dirichlet vector 
$S$ given at~\eq{eq:sdir},  independent of $\{B_i\}_{i\geq1}\sim\CRP\bclr{\frac{\beta}{\alpha}, \frac{\theta + \beta K_\emm}{\alpha}}$.
We will show that jointly: $T_j \eqd \sum_{i: M_i = j}P_i$ for $j=1,\ldots, K_\emm$, and the reranking in decreasing order of
$(\sum_{i: M_i =0, i \in B_j} P_i)_{j\geq1}$ has the same distribution as
the analogous reranking of  $\clr{T_0 \wt P_i}_{i\geq 1}$, from which the result follows
from \cite[Corollary~20]{Pitman1996} stated above.

To show the claimed joint equality in distribution, we use \cite[Proposition~21]{Pitman1997a}: 
For $0<\alpha< 1$, let $\tau$ be a subordinator with L\'evy measure $\Gamma(1-\alpha)^{-1}\alpha x^{-\alpha -1} e^{-x} dx$ independent of the 
standard gamma subordinator $(\gamma(s))_{s\geq0}$. Fixing $\tilde \theta>0$ and writing $\zeta_1>\zeta_2>\cdots$ for the ranked jump sizes
  of the subordinator $\tau$ appearing in the in the interval $(0,\gamma(\tilde \theta/\alpha))$, we have
\be{
\frac{1}{\tau\bclr{\gamma(\tilde \theta/\alpha)}} \bclr{\zeta_1, \zeta_2,\ldots} \sim \PD(\alpha, \tilde \theta),
}
and is independent of $\tau\bclr{\gamma(\tilde \theta/\alpha)} \sim \mathrm{Gamma}(\tilde \theta, 1)$. 
A closely related fact that we use below is that $\bclr{\tau\bclr{\gamma(s/\alpha)}}_{s\geq0}$ is a standard gamma subordinator.

We apply this construction with $\tilde \theta =\theta+\emm$, and condition on $\{C_1,\ldots, C_{K_\emm}\}$, which is used in both the construction of $T$ and $S$ (and others).
We then write $t_j=\alpha^{-1}\sum_{i=1}^j  (\abs{C_i}-\beta)$ for $j=1,\ldots, K_\emm$ and 
$t=t_{K_\emm}+\alpha^{-1} (\theta+K_\emm\beta)$. 
Specializing the construction above of \cite[Proposition~21]{Pitman1997a}, if $\zeta_1>\zeta_2>\cdots$ are the ranked jump sizes
  of the subordinator $\tau$ appearing in the in the interval $(0,\gamma(t))$, then 
\be{
\frac{1}{\tau\bclr{\gamma(t)}} \bclr{\zeta_1, \zeta_2,\ldots} \sim \PD(\alpha, \theta +\emm ),
}
and we use the variables on the left hand side to construct the coagulation operation of Definition~\ref{def:PDGMcoagfrag}.
The beta-gamma algebra implies that we can represent the Dirichlet vector $S$ (given at~\eq{eq:sdir} in the coagulation operation of Definition~\ref{def:PDGMcoagfrag}) by
\be{
S =\frac{1}{\gamma(t)} \bclr{ \gamma(t)-\gamma(t_{K_\emm}), \gamma(t_1), \gamma(t_2)-\gamma(t_1),\ldots, \gamma(t_{K_\emm})-\gamma(t_{K_\emm-1})},
}
which is
independent of $\gamma(t)$. Similarly, and using the fact that $(\tau\clr{\gamma(s/\alpha)})_{s\geq0}$ is a standard gamma subordinator, we can represent
\be{
T=\frac{\bclr{ \tau\bclr{\gamma(t)}-\tau\bclr{\gamma(t_{K_\emm})},\tau\bclr{ \gamma(t_1)}, \tau\bclr{\gamma(t_2)}-\tau\bclr{\gamma(t_1)},\ldots, \tau\bclr{\gamma(t_{K_\emm})}-\tau\bclr{\gamma(t_{K_\emm-1})}}}{\tau\bclr{\gamma(t)}} 
}
which is independent of $\tau\bclr{\gamma(t)}$.

Now, conditioning on $S$ and considering the coagulation operation of Definition~\ref{def:PDGMcoagfrag},  
because the positions of jumps of $\tau$ in the interval $(0,\gamma(t))$ are conditionally i.i.d.\ distributed $\mathrm{Uniform}(0,\gamma(t))$, the probability that a jump is in the interval $\clr{\gamma(t_{j-1}), \gamma(t_j)}$ is $S_j$ for $j=1,\ldots, K_\emm$ and the probability it is in 
in the interval $\clr{\gamma(t)- \gamma(t_{K_\emm}), \gamma(t)}$ is $S_0$. Thus we can realise the $M_i$ variables via the location of jumps. For jumps occurring at times in intervals of the form $\clr{\gamma(t_{j-1}), \gamma(t_j)}$,
 the coagulation operation  sums such jumps together, which gives a mass of  total length $(\tau(\gamma\clr{t_j})- \tau(\gamma\clr{t_{j-1}}))/\tau\clr{\gamma(t)}= T_j$, $j=1,\ldots, K_\emm$, which implies part of the  equality in joint distribution we need to show. 
 
 For the remaining jumps of $\tau$
 in the interval  $\clr{\gamma(t)- \gamma(t_{K_\emm}), \gamma(t)}$, the coagulation operation combines jumps according to a $\CRP\bclr{\frac{\beta}{\alpha}, \frac{\theta+ K_\emm \beta}{\alpha}}$ applied to their indices (in any order, by exchangeability).
 Thus, by beta-gamma algebra and using the expression for $T_0$, the final equality we need to show is that
 the jumps of $\tau$ in the interval  $\clr{\gamma(t)- \gamma(t_{K_\emm}), \gamma(t)}$ divided by $(\tau(\gamma(t))-\tau(\gamma(t_{K_\emm})))$ and  coagulated by 
 the $\CRP\bclr{\frac{\beta}{\alpha}, \frac{\theta+ K_\emm \beta}{\alpha}}$ has the same distribution as $(\wt P_i)_{i\geq1}\sim\PD(\beta, \theta+K_\emm\beta)$ after reranking, as  in our description of \cite[Corollary~20]{Pitman1999b}.
 But using independent increments,
 \cite[Proposition~21]{Pitman1997a} described above implies that these  
 normalized jumps are distributed (reranked) as $\PD(\alpha, \theta+K_\emm \beta)$,
 and the desired result then follows from the $\emm=0$ case of the duality result from \cite[Theorem~12]{Pitman1999b}. 
\end{proof}

\end{document}